\documentclass[12pt]{amsart}
\usepackage{amsmath}
\usepackage{tikz}
\usetikzlibrary{trees}
\usepackage{amsxtra}
\usepackage{enumitem}
\usepackage{amscd}
\usepackage{tikz-cd}
\usepackage{amsthm}
\usepackage{hyperref}
\usepackage[T1]{fontenc}
\usepackage{amsfonts}
\usepackage{amssymb}
\usepackage{eucal}
\usepackage[all]{xypic}
\usepackage{color} 
\newtheorem{Theo}[subsubsection]{Theorem}
\newtheorem{Theor}{Theorem}
\newenvironment{hproof}{%
  \proof}{\endproof}
\newtheorem{Theore}{Theorem}

\newtheorem{cor}[subsubsection]{Corollary}

\theoremstyle{definition}

\newtheorem{Coro}{Corollary}
\newtheorem{Corol}{Corollary}

\newtheorem{fact}[subsubsection]{Fact} 
\newtheorem{rem}[subsubsection]{Remark}

\newtheorem{question}[subsubsection]{Question}
\newtheorem{Prop}[subsubsection]{Proposition}
\newtheorem{Obs}[subsubsection]{Observation}
\newtheorem{example}[subsubsection]{Example}

\newtheorem{Lemm}[subsubsection]{Lemma} 

\newtheorem{definition}[subsubsection]{Definition}

\newtheorem{ex}[subsection]{Exercise}

\usepackage{yhmath}
\DeclareSymbolFont{largesymbols}{OMX}{yhex}{m}{n}
\DeclareMathAccent{\widetilde}{\mathord}{largesymbols}{"65}

\newcommand{\bp}{\begin{Prop}}
\newcommand{\ep}{\end{Prop}}

\newcommand{\bl}{\begin{Lemm}}
\newcommand{\el}{\end{Lemm}}
\newcommand{\bex}{\begin{ex} \rm}
\newcommand{\eex}{\end{ex}}
\newcommand{\bt}{\begin{Theo}}
\newcommand{\et}{\end{Theo}}
\newcommand{\bq}{\begin{question}}
\newcommand{\eq}{\end{question}}

\newcommand{\bc}{\begin{cor}}
\newcommand{\ec}{\end{cor}}
\newcommand{\bob}{\begin{Obs}}
\newcommand{\eob}{\end{Obs}}

\newcommand{\nc}{\newcommand}
\nc{\renc}{\renewcommand}
\nc{\ssec}{\subsection}
\nc{\sssec}{\subsubsection} 
\nc\ol{\overline}
\nc\wt{\widetilde}
\nc\wh{\widehat}
\nc\tboxtimes{\wt{\boxtimes}}

\renc{\d}{{\delta}}
\nc{\Aa}{{\mathbb{A}}}
\nc{\Bb}{{\mathbb{B}}}
 \nc{\Gg}{{\mathbb{G}}}  
\nc{\Hh}{{\mathbb{H}}}
 \nc{\Nn}{{\mathbb{N}}}
\nc{\Pp}{{\mathbb{P}}}
\nc{\Rr}{{\mathbb{R}}}
\newcommand{\F}{\mathbb{F}}
\nc{\BV}{{\mathbb{V}}}
\nc{\BW}{{\mathbb{W}}}
\newcommand{\Z}{\mathbb{Z}}
\newcommand{\N}{\mathbb{N}}

\nc{\Qq}{{\mathbb{Q}}}
\nc{\Ss}{{\mathbb{S}}}
\nc{\Cc}{{\mathbb{C}}}
\nc{\Ff}{{\mathbb{F}}}
 \nc{\EL}{{L_\infty}}

\nc{\CA}{{\mathcal{A}}}
\nc{\CB}{{\mathcal{B}}}

\nc{\CE}{{\mathcal{E}}}
\nc{\CF}{{\mathcal{F}}}

\nc{\Las}{\mathsf{Las}}
\nc{\CG}{{\mathcal{G}}}

\nc{\CL}{{\mathcal{L}}}
\nc{\CC}{{\mathcal{C}}}
\nc{\CM}{{\mathcal{M}}}

\nc{\CN}{{\mathcal{N}}}
\nc{\Oog}{{\mathbb{O}}}
\nc{\Oo}{{\mathcal{O}}}
\nc{\CP}{{\mathcal{P}}}
\nc{\CQ}{{\mathcal{Q}}}
\nc{\CR}{{\mathcal{R}}}
\nc{\CS}{{\mathcal{S}}}
\nc{\CT}{{\mathcal{T}}}
\nc{\CU}{{\mathcal{P}}}

\nc{\CV}{{\mathcal{V}}}
 
\nc{\CW}{{\mathcal{W}}}
\nc{\CZ}{{\mathcal{Z}}}

\nc{\cM}{{\check{\mathcal M}}{}}
\nc{\csM}{{\check{\mathcal A}}{}}
\nc{\oM}{{\overset{\circ}{\mathcal M}}{}}
\nc{\obM}{{\overset{\circ}{\mathbf M}}{}}
\nc{\oCA}{{\overset{\circ}{\mathcal A}}{}}
\nc{\obA}{{\overset{\circ}{\mathbf A}}{}}
\nc{\ooM}{{\overset{\circ}{M}}{}}
\nc{\osM}{{\overset{\circ}{\mathsf M}}{}}
\nc{\vM}{{\overset{\bullet}{\mathcal M}}{}}
\nc{\nM}{{\underset{\bullet}{\mathcal M}}{}}
\nc{\oD}{{\overset{\circ}{\mathcal D}}{}}
\nc{\obD}{{\overset{\circ}{\mathbf D}}{}}
\nc{\oA}{{\overset{\circ}{\mathbb A}}{}}
\nc{\op}{{\overset{\bullet}{\mathbf p}}{}}
\nc{\cp}{{\overset{\circ}{\mathbf p}}{}}
\nc{\oU}{{\overset{\bullet}{\mathcal U}}{}}
\nc{\oZ}{{\overset{\circ}{\mathcal Z}}{}}
\nc{\ofZ}{{\overset{\circ}{\mathfrak Z}}{}}
\nc{\oF}{{\overset{\circ}{\fF}}}

\nc{\fa}{{\mathfrak{a}}}
\nc{\fb}{{\mathfrak{b}}}
\nc{\fg}{{\mathfrak{g}}}
\nc{\fgt}{{\fg}_!}
\nc{\fgl}{{\mathfrak{gl}}}
\nc{\fh}{{\mathfrak{h}}}
\nc{\fj}{{\mathfrak{j}}}
\nc{\fm}{{\mathfrak{m}}}
\nc{\ft}{{\mathfrak{t}}}
\nc{\fn}{{\mathfrak{n}}}
\nc{\fu}{{\mathfrak{u}}}
\nc{\fp}{{\mathfrak{p}}}
\nc{\fr}{{\mathfrak{r}}}
\nc{\fs}{{\mathfrak{s}}}
\nc{\fsl}{{\mathfrak{sl}}}
\nc{\hsl}{{\widehat{\mathfrak{sl}}}}
\nc{\hgl}{{\widehat{\mathfrak{gl}}}}
\nc{\hg}{{\widehat{\mathfrak{g}}}}
\nc{\chg}{{\widehat{\mathfrak{g}}}{}^\vee}
\nc{\hn}{{\widehat{\mathfrak{n}}}}
\nc{\chn}{{\widehat{\mathfrak{n}}}{}^\vee}

\nc{\fA}{{\mathfrak{A}}}
\nc{\fB}{{\mathfrak{B}}}
\nc{\fD}{{\mathfrak{D}}}
\nc{\fE}{{\mathfrak{E}}}
\nc{\fF}{{\mathfrak{F}}}
\nc{\fG}{{\mathfrak{G}}}
\nc{\fK}{{\mathfrak{K}}}
\nc{\fL}{{\mathfrak{L}}}
\nc{\fM}{{\mathfrak{M}}}
\nc{\fN}{{\mathfrak{N}}}
\nc{\fP}{{\mathfrak{P}}}
\nc{\fU}{{\mathfrak{U}}}
\nc{\fV}{{\mathfrak{V}}}
\nc{\fZ}{{\mathfrak{Z}}}
\newcommand{\Q}{\mathbb{Q}}
\nc{\bb}{{\mathbf{b}}}
\nc{\bd}{\partial}
\nc{\be}{{\mathbf{e}}}
\nc{\bj}{{\mathbf{j}}}
\nc{\bn}{{\mathbf{n}}}
\nc{\bF}{{\mathbf{F}}}
\nc{\bu}{{\mathbf{u}}}
\nc{\bv}{{\mathbf{v}}}
\nc{\bx}{{\mathbf{x}}}
\nc{\bs}{{\mathbf{s}}}
\nc{\by}{{\bar{y}}}
\nc{\bw}{{\mathbf{w}}}
\nc{\bA}{{\mathbf{A}}}
\nc{\bK}{{\mathbf{K}}}
\nc{\bI}{{\mathbf{I}}}
\nc{\bB}{{\mathbf{B}}}
\nc{\bG}{{\mathbf{G}}}

\nc{\bD}{{\mathbf{D}}}
\nc{\bP}{{\mathbf{P}}}
\nc{\bH}{{\mathbf{H}}}
\nc{\bM}{{\mathbf{M}}}
\nc{\bN}{{\mathbf{N}}}
\nc{\bV}{{\mathbf{V}}}
\nc{\bU}{{\mathbf{U}}}
\nc{\bL}{{\mathbf{L}}}

\nc{\bW}{{\mathbf{W}}}
\nc{\bX}{{\mathbf{X}}}
\nc{\bY}{{\mathbf{Y}}}
\nc{\bZ}{{\mathbf{Z}}}
\nc{\bS}{{\mathbf{S}}}
\nc{\bSi}{{\bar{\Sigma}}}
\nc{\sA}{{\mathsf{A}}}
\nc{\sB}{{\mathsf{B}}}
\nc{\sC}{{\mathsf{C}}}
\nc{\sD}{{\mathsf{D}}}
\nc{\sF}{{\mathsf{F}}}
\nc{\sG}{{\mathsf{G}}}
\nc{\sK}{{\mathsf{K}}}
\nc{\sM}{{\mathsf{M}}}
\nc{\sO}{{\mathsf{O}}}
\nc{\sQ}{{\mathsf{Q}}}
\nc{\sP}{{\mathsf{P}}}
\nc{\sZ}{{\mathsf{Z}}}
\nc{\sfp}{{\mathsf{p}}}
\nc{\sr}{{\mathsf{r}}}
\nc{\sg}{{\mathsf{g}}}
\nc{\sff}{{\mathsf{f}}}
\nc{\sfb}{{\mathsf{b}}}
\nc{\sfc}{{\mathsf{c}}}
\nc{\sd}{{\ltimes}} 

\nc{\tH}{{\widetilde{H}}}
\nc{\tA}{{\widetilde{\mathbf{A}}}}
\nc{\tB}{{\widetilde{\mathcal{B}}}}
\nc{\tg}{{\widetilde{\mathfrak{g}}}}
\nc{\tG}{{\widetilde{G}}}

\nc{\TM}{{\widetilde{\mathbb{M}}}{}}
\nc{\tO}{{\widetilde{\mathsf{O}}}{}}
\nc{\tU}{\widetilde{U}}
\nc{\TZ}{{\tilde{Z}}}
\nc{\tx}{{\tilde{x}}}
\nc{\tq}{{\tilde{q}}}
\nc{\R}{\mathbb{R}}

\nc{\tfP}{{\widetilde{\mathfrak{P}}}{}}
\nc{\tz}{{\tilde{\zeta}}}
\nc{\tmu}{{\tilde{\mu}}}

  \nc{\vol}{{\mathop{\operatorname{\rm vol\,}}}}
  
  \nc{\gal}{{\mathop{\operatorname{\rm Gal\,}}}}
  \nc{\cl}{{\mathop{\operatorname{\rm cl}}}}
  \nc{\disc}{{\mathop{\operatorname{\rm disc}}}}
  \nc{\Sym}{{\mathop{\operatorname{\rm Sym}}}}
   \nc{\Aut}{{\mathop{\operatorname{\rm Aut}}}}
 \nc{\Spec}{{\mathop{\operatorname{\rm Spec}}}}
  \nc{\spec}{{\mathop{\operatorname{\rm Spec}}}}
\nc{\Ker}{{\mathop{\operatorname{\rm Ker}}}}
 \nc{\dom}{{\mathop{\operatorname{\rm dom}}}}
\nc{\End}{{\mathop{\operatorname{\rm End}}}}
 \nc{\Hom}{\operatorname{\Hom}}
 \nc{\GL}{{\mathop{\operatorname{\rm GL}}}}
 \nc{\Id}{{\mathop{\operatorname{\rm Id}}}}
 \nc{\rk}{{\mathop{\operatorname{\rm rk}}}}
 \nc{\length}{{\mathop{\operatorname{\rm length}}}}
\nc{\supp}{{\mathop{\operatorname{\rm supp} \, }}}
\nc{\val}{{\rm val}}
\nc{\res}{{\mathop{\operatorname{\rm res}}}}

\def\Ind#1#2#3{{#1} {\downarrow}_{#3} {#2} }

\nc{\seq}[1]{\stackrel{#1}{\sim}}

\def\beq#1{\begin{equation} \label{ #1}}
\def\eeq{\end{equation}}

\def\prf{\begin{proof}}
\def\pv{\end{proof} }
 \def\eprf{\end{proof} }

 \renc{\b}{{\beta}}

\def\Ind#1#2{#1\setbox0=\hbox{$#1x$}\kern\wd0\hbox to 0pt{\hss$#1\mid$\hss}
\lower.9\ht0\hbox to 0pt{\hss$#1\smile$\hss}\kern\wd0}

\usepackage{color}
\usepackage{hyperref}

\usepackage{url}

 \makeatletter
\def\@setthanks{\vspace{-\baselineskip}\def\thanks##1{\@par##1\@addpunct.}\thankses}
\makeatother
\title{Decidability via the tilting correspondence}
\author{Konstantinos Kartas}

\thanks{During this research, the author was funded by EPSRC grant EP/20998761 and was also supported by the Onassis Foundation - Scholarship ID: F ZP 020-1/2019-2020.}

\newcommand{\Addresses}{{
  \bigskip
  \footnotesize

\textsc{Mathematical Institute, Woodstock Road, Oxford OX2 6GG.}\par\nopagebreak
  \textit{E-mail address}: \texttt{kartas@maths.ox.ac.uk}
}
}

\begin{document}
\maketitle

\begin{abstract}
We prove a relative decidability result for perfectoid fields. This applies to show that the fields $\Q_p(p^{1/p^{\infty}})$ and $\Q_p(\zeta_{p^{\infty}})$ are (existentially) decidable relative to the perfect hull of $ \F_p(\!(t)\!)$ and $\Q_p^{ab}$ is (existentially) decidable relative to the perfect hull of $\overline{ \F}_p(\!(t)\!)$. We also prove some unconditional decidability results in mixed characteristic via reduction to characteristic $p$.
\end{abstract}
\setcounter{tocdepth}{1}
\tableofcontents
\section*{Introduction} 

After the decidability of the $p$-adic numbers $\Q_p$ was established by Ax-Kochen \cite{AK} and independently by Ershov \cite{Ershov}, several decidability questions about local fields and their extensions have been raised and answered:
\begin{itemize}
\item In mixed characteristic, Kochen \cite{Kochen} showed that $\Q_p^{ur}$, the maximal unramified extension of $\Q_p$, is decidable in the language $L_{\text{val}}=\{0,1,+,\cdot,\Oo\}$ (see notation). More generally, by work of \cite{Ziegler}, \cite{Ershov}, \cite{Bas2}, \cite{Bel} and more recently \cite{AJ}, \cite{Junguk} and \cite{Lee}, we have a good understanding of the model theory of unramified and finitely ramified mixed characteristic henselian fields. 
\item In positive characteristic, our understanding is much more limited. Nevertheless, by work of Denef-Schoutens \cite{Den}, we know that $\F_p(\!(t)\!)$ is \textit{existentially} decidable in $L_{\text{val}}(t)=\{0,1,t,+,\cdot,\Oo\}$ (see notation), modulo resolution of singularities.\footnote{ \label{definabilityfoot} The formalism of \cite{Den} does not include a unary predicate $\Oo$ for the valuation ring. This does not make a difference because of the equivalences $x\in \Oo \leftrightarrow \exists y(y^2=1+t\cdot x^2)$ (for $p>2$; replace squares with cubes for $p=2$) and $x\notin \Oo \leftrightarrow x^{-1}\in t\cdot \Oo$.} In fact, Theorem 4.3 \cite{Den} applies to show that (assuming resolution) any finitely ramified extension of $\F_p(\!(t)\!)$ is existentially decidable relative to its residue field. Anscombe-Fehm \cite{AnscombeFehm} showed \textit{unconditionally} that $\F_p(\!(t)\!)$ is existentially decidable but in the language $L_{\text{val}}$, which does not include a constant symbol for $t$.
\end{itemize}

Our understanding is less clear for \textit{infinitely ramified} fields and there are many extensions of $\Q_p$ (resp. $\F_p(\!(t)\!)$) of great arithmetic interest, whose decidability problem is still open: 
\begin{itemize}
\item 
In mixed characteristic, these include $\Q_p^{ab}$, the maximal abelian extension of $\Q_p$ and the totally ramified extension $\Q_p(\zeta_{p^{\infty}})$, obtained by adjoining all $p^n$-th roots of unity. These extensions had already been discussed in Macintyre's survey on pg.140 \cite{Mac} and a conjectural axiomatization of $\Q_p^{ab}$ was given by Koenigsmann on pg.55 \cite{Koen}. Another interesting extension is $\Q_p(p^{1/p^{\infty}})$, a totally ramified extension of $\Q_p$ obtained by adjoining a compatible system of $p$-power roots of $p$. 
\item In positive characteristic, two very natural infinitely ramified fields are the perfect hulls of $\F_p(\!(t)\!)$ and $\overline{\F}_p(\!(t)\!)$. Both of these fields have been conjectured to be decidable (see pg. 4 \cite{KuhlBla} for the former). The recent work of Kuhlmann-Rzepka \cite{KuhlBla} ultimately aims at extending Kuhlmann's earlier results for tame fields to cover fields like the ones mentioned above. 
\end{itemize}
The above fields will be the main objects of interest throughout the paper. Their $p$-adic (resp. $t$-adic) completions are typical examples of \textit{perfectoid fields} in the sense of Scholze \cite{Scholze} (see \S \ref{localapprox}). A perfectoid field $(K,v)$ is a valued field which is complete with respect to a non-discrete valuation of rank 1, with residue characteristic equal to $p>0$ and such that the Frobenius map $\Phi:\Oo_K/(p)\to \Oo_K/(p):x\mapsto x^p$ is surjective. Loosely speaking, the last condition says that one can extract approximate $p$-power roots of any element in the field. For any perfectoid field $K$, one can define its tilt $K^{\flat}$ (see \S \ref{perfsec}), which intuitively is its \textit{local} function field analogue and serves as a good characteristic $p$ approximation of $K$. In practice, this means that one can often reduce arithmetic problems about $K$ to arithmetic problems about $K^{\flat}$. This kind of transfer principle, which works for a fixed $p$, should be contrasted with the Ax-Kochen/Ershov principle which achieves such a reduction only \textit{asymptotically}, i.e. with the residue characteristic $p\to \infty$. This is explained in detail in Scholze's ICM report (see pg.2 \cite{ScholzeICM}).\\

Our main goal is to set the stage for incorporating ideas from perfectoid geometry in the model theory of henselian fields. In the present paper we focus on decidability, although it is conceivable that our methods can be used in different model-theoretic contexts. 
We will prove the following relative decidability result for the fields discussed above:
\begin{Coro} \label{mainsCor}
$(a)$ Assume $\F_p(\!(t)\!)^{1/p^{\infty}}$ is decidable (resp. $\exists$-decidable) in $L_{\text{val}}(t)$. Then $\Q_p(p^{1/p^{\infty}})$ and $\Q_p(\zeta_{p^{\infty}})$ are decidable (resp. $\exists$-decidable) in $L_{\text{val}}$.\\
$(b)$ Assume $\overline{ \F}_p(\!(t)\!)^{1/p^{\infty}}$ is decidable (resp. $\exists$-decidable) in $L_{\text{val}}(t)$. Then $\Q_p^{ab}$ is decidable (resp. $\exists$-decidable) in $L_{\text{val}}$.
\end{Coro}
As usual, we write $\F_p(\!(t)\!)^{1/p^{\infty}}$ (resp. $\overline{ \F}_p(\!(t)\!)^{1/p^{\infty}}$) for the perfect hull of $\F_p(\!(t)\!)$ (resp. $\overline{\F}_p(\!(t)\!)$), $L_{\text{val}}$ for the language of valued fields and $L_{\text{val}}(t)$ for the language $L_{\text{val}}$ enriched with a constant symbol for $t$ (see notation). Corollary \ref{mainCor} is essentially a special case of a general relative decidability result for perfectoid fields, which is discussed next. (Strictly speaking, the fields in Corollary \ref{mainsCor} are not perfectoid but one can still derive Corollary \ref{mainsCor} from Theorem \ref{reldec} directly.)

Let $F$ be a perfectoid field of characteristic $p$ (e.g., $F=\widehat{\F_p(\!(t)\!)^{1/p^{\infty}}}$). An \textit{untilt} of $F$ is a mixed characteristic perfectoid field $K$ together with an isomorphism $\iota:K^{\flat}\stackrel{\cong}\to F$. In general, there will be many non-isomorphic untilts of $F$. In fact, there will be too many untilts even up to elementary equivalence (see Proposition \ref{scholzprop}$(b)$), thus shattering the naive guess of $K$ being decidable relative to its tilt $K^{\flat}$. On a more elementary level, ones needs to assume decidability with \textit{parameters} on the positive characteristic side (see Example \ref{lang}). Nevertheless, a relative decidability result for perfectoids can be salvaged by asking that $K$ be an $R_0$-\textit{computable} untilt of $F$, a notion which is briefly explained below (see \S \ref{compunt} for details).

Write $W(\Oo_F)$ for the ring of Witt vectors over $\Oo_F$ (see \S \ref{Witt}). An element $\xi \in W(\Oo_F)$ is identified with its associated Witt vector, which is an infinite sequence $(\xi_0,\xi_1,...)$ with $\xi_n\in \Oo_F$. Whenever there is a \textit{computable} subring $R_0\subseteq \Oo_F$ such that $\xi_n\in R_0$ and the function $\N\to R_0:n\mapsto \xi_n$ is \textit{recursive}, we say that $\xi$ is $R_0$-computable. As usual, a computable ring $R_0$ is one whose underlying set is (or may be indetified with) a recursive subset of $\N$, so that the ring operations are (or are identified with) recursive functions (e.g., $\F_p[t]$ is a computable ring). An important result by Fargues-Fontaine gives a 1-1 correspondence between untilts of $F$ and certain principal ideals of $W(\Oo_F)$ (Theorem \ref{Font}). An untilt $K$ of $F$ is then said to be an $R_0$-\textit{computable untilt} whenever there exists an $R_0$-computable generator $\xi \in W(\Oo_F)$ of the ideal corresponding to $K$. We then have the following relative decidability result (the name is due to Scanlon):
\begin{Theore} [Perfectoid transfer]\label{reldec}
Suppose $K$ is an $R_0$-computable untilt of $F$. If $F$ is decidable in $L_{\text{val}}(R_0)$, then $K$ is decidable in $L_{\text{val}}$.
\end{Theore}
Here $L_{\text{val}}(R_0)$ is the language $L_{\text{val}}$ enriched with constant symbols for the elements of $R_0$. We also obtain an \textit{existential} version of Theorem \ref{reldec} in \S \ref{mainthmsec}. The condition on $K$ being an $R_0$-computable untilt of $F$ is true in virtually all cases of interest (although generically false), even for a natural choice of $R_0$. Corollary \ref{mainCor} is an immediate consequence of Theorem \ref{reldec}, by considering the $t$-adic completion of $\F_p(\!(t)\!)^{1/p^{\infty}}$ (resp. $\overline{ \F}_p(\!(t)\!)^{1/p^{\infty}}$ for part $(b)$) and $R_0=\F_p[t]$. The two key ingredients in the proof of Theorem \ref{reldec} are: 
\begin{enumerate}
\item An (unpublished) Ax-Kochen/Ershov style result by van den Dries for general mixed characteristic henselian fields, presented in \S \ref{sec1}.
\item The Fargues-Fontaine correspondence described above, presented in \S \ref{Untilting}.
\end{enumerate}
Let us now outline how these two ingredients will be combined towards the proof of Theorem \ref{reldec}. The result by van den Dries will enable us to reduce the decidability of $K$ to the \textit{uniform} decidability of the sequence of residue rings $(\Oo_K/(p^n))_{n\in \omega}$ and also the decidability of $(\Gamma_v,vp)$. The Fargues-Fontaine correspondence then gives us a way to interpret each individual residue ring $\Oo_K/(p^n)$ inside the tilt $F$ (with parameters from $R_0$), via an interpretation $E_n$. The assumption on $K$ being $R_0$-computable will then imply that the sequence of interpretations $(E_n)_{n\in \omega}$ is \textit{uniformly recursive} (see \S \ref{uniformrec}). Having assumed the decidability of $F$ in $L_{\text{val}}(R_0)$, this will then give us that the sequence $(\Oo_K/(p^n))_{n\in \omega}$ is uniformly decidable. The pointed value group $(\Gamma_v,vp)$ is easily seen to be interpretable in $F$ (with parameters from $R_0$) and the decidability of $K$ follows.  

It should be mentioned that---prior to this work---Rideau-Scanlon-Simon had made a formative observation, namely that the Fargues-Fontaine correspondence already provides us with an interpretation of the valuation ring $\Oo_K$ in $\Oo_F$ (with parameters) in the sense of continuous logic. This however does not yield an interpretation in the sense of ordinary first-order logic (see \S \ref{rideauscanlon}). On the other hand, there do exist honest first-order interpretations of the \textit{truncated} residue rings $\Oo_K/(p^n)$ in $\Oo_F$, which is why it is convenient to work with them instead.

Theorem \ref{reldec} allows us to reduce decidability problems from the mixed characteristic to the positive characteristic world. Although our model-theoretic understanding of the latter is notoriously limited, this method still yields several applications. In \S \ref{tamesec} and \S \ref{cong} we obtain new \textit{unconditional} decidability results in mixed characteristic via reduction to characteristic $p$. In \S \ref{tamesec}, we prove the following:
\begin{Coro} \label{mainTame}
The valued field $(\Q_p(p^{1/p^{\infty}}),v_p)$ (resp. $(\Q_p(\zeta_{p^{\infty}}),v_p)$) admits a maximal immediate extension which is decidable in $L_{\text{val}}$.
\end{Coro}
Note that the fields $\Q_p(p^{1/p^{\infty}})$ and $\Q_p(\zeta_{p^{\infty}})$ are \textit{not} Kaplansky and have many non-isomorphic maximal immediate extensions, all of which are \textit{tame} in the sense of Kuhlmann \cite{Kuhl} (see Example \ref{tamexample}). Although the work of Kuhlmann yields several decidability results for equal characteristic tame fields, Corollary \ref{mainTame} is, to my knowledge, the first decidability result for tame fields of mixed characteristic (see \S \ref{tamesec}). The proof uses a recent decidability result of Lisinski \cite{Lis} for the Hahn field $\F_p(\!(t^{\Gamma})\!)$ with $\Gamma=\frac{1}{p^{\infty}}\Z$ in the language $L_{\text{val}}(t)$, strengthening Kuhlmann's earlier result for $\F_p(\!(t^{\Gamma})\!)$ in the language $L_{\text{val}}$ (see Theorem 1.6 \cite{Kuhl}). Corollary \ref{mainTame} then follows from Theorem \ref{reldec} and basic properties of the tilting equivalence (see \S \ref{tiltingequivsec}). It is worth remarking that most maximal immediate extensions of $\Q_p(p^{1/p^{\infty}})$ (resp. $\Q_p(\zeta_{p^{\infty}})$) are \textit{undecidable} (see Remark \ref{mostmaxundec}).

Without making essential use of the perfectoid machinery but only the philosophy thereof, we show in \S \ref{cong} the following:
\begin{Theore} \label{cong1}
Let $K$ be any of the valued fields $\Q_p(p^{1/p^{\infty}}),\Q_p(\zeta_{p^{\infty}})$ or $\Q_p^{ab}$. The existential theory of $\Oo_K/p \Oo_K$ in the language of rings $L_{\text{rings}}$ is decidable.
\end{Theore}
The proof is again via reduction to characteristic $p$, using a recent existential decidability result for equal characteristic henselian valued fields in the language $L_{\text{val}}$, due to Anscombe-Fehm \cite{AnscombeFehm}. However, as Corollary \ref{mainsCor} suggests, if we aim to understand the theories of $\Q_p(p^{1/p^{\infty}}),\Q_p(\zeta_{p^{\infty}})$ and $\Q_p^{ab}$ via reduction to positive characteristic, we will need stronger results in the language $L_{\text{val}}(t)$ on the characteristic $p$ side. This is also supported by Proposition \ref{from0top} which shows that the Diophantine problem for $\F_p(\!(t)\!)^{1/p^{\infty}}$ in $L_{\text{val}}(t)$ is Turing reducible to the $\forall^1 \exists$-theory of $\Q_p(p^{1/p^{\infty}})$. 

An application of a different flavor, which yields an \textit{undecidability} result in mixed characteristic via reduction to characteristic $p$, was recently found in \cite{KK3}. Theorem A \cite{KK3} shows that the asymptotic theory of $\{ K: [K:\Q_p]<\infty\}$ in the language $L_{\text{val}}$ with a cross-section is undecidable. 
\section*{Notation}

\begin{itemize}
\item If $(K,v)$ is a valued field, we denote by $\mathcal{O}_v$ the valuation ring. If the valuation is clear from the context, we shall also write $\mathcal{O}_K$. We write $\Gamma_v$ for the value group and $k_v$ for the residue field. If the valuation $v$ is clear from the context, we also denote them by $\Gamma$ and $k$ respectively. 
\item When $(K,v)$ is of mixed characteristic, we write $p$ for the number $\text{char}(k)$. The notation $p^n \Oo_v$ stands for the ideal of $\Oo_v$ generated by the element $p^n$. If both the field in question and the valuation $v$ are clear from the context, we shall write $(p^n)$ for $p^n \Oo_v$.  
\item $\Z_p^{ur}$: The valuation ring of $(\Q_p^{ur},v_p)$, the maximal unramified extension of $\Q_p$ equipped with the unique extension of the $p$-adic valuation.\\
$\Z_p^{ab}$: The valuation ring of $(\Q_p^{ab},v_p)$, the maximal Galois extension of $\Q_p$ whose Galois group over  $\Q_p$ is abelian.\\
We write $\F_p(\!(t)\!)^{1/p^{\infty}}$ (resp. $\overline{ \F}_p(\!(t)\!)^{1/p^{\infty}}$) for the perfect hull of $\F_p(\!(t)\!)$ (resp. $\overline {\F}_p(\!(t)\!)$).
\item For a given language $L$, we denote by $\text{Sent}_L$ the set of $L$-sentences and by $\text{Form}_L$ the set of $L$-formulas. If $M$ is an $L$-structure and $A\subseteq M$ is an arbitrary subset, we write $L(A)$ for the language $L$ enriched with a constant symbol $c_a$ for each element $a \in A$. The $L$-structure $M$ can be updated into an $L(A)$-structure in the obvious way.

\item We write $L_{\text{rings}}=\{0,1,+,\cdot\}$ for the language of rings, $L_{\text{oag}}=\{0,+,<\}$ for the language of ordered abelian groups, $L_{\text{val}}=L_{\text{rings}}\cup \{\Oo\}$ for the language of valued fields (where $\Oo$ is a unary predicate for the valuation ring), 
$L_{\text{val}}(t)=L_{\text{val}}\cup \{t\}$, where $t$ is a constant symbol whose intended interpretation will always be clear from the context. A local ring $(R,\mathfrak{m}_R)$ may be viewed as an $L_{\text{lcr}}$-structure, where $L_{\text{lcr}}$ is the \textit{language of local rings} consisting of the language of rings $L_{\text{rings}}$ together with a unary predicate $\mathfrak{m}$, whose intended interpretation is the maximal ideal $\mathfrak{m}_R\subseteq R$. 
\end{itemize}
\section{Preliminaries}

\subsection{Decidability}
\subsubsection{Introduction} Fix a \textit{countable} language $L$ and let $\text{Sent}_L$ be the set of well-formed $L$-sentences, identified with $\N$ via some G\"odel numbering. 
Let $M$ be an $L$-structure. Recall that $M$ is decidable if we have an algorithm to decide whether $M\models \phi$, for any given $\phi \in \text{Sent}_L$. More formally, let $\chi_{M}:\text{Sent}_L \to \{0,1\}$ be the characteristic function of $\text{Th}(M)\subseteq \text{Sent}_L$. 
We say that $M$ is \textit{decidable} if $\chi_M$ is recursive.
\subsubsection{Uniform decidability}
\begin{definition} \label{uniformdef}
For each $n\in \N$, let $f_n:\N\to \N$ be a function. The sequence $(f_n)_{n\in \omega}$ is uniformly recursive if the function $\N\times \N\to \N: (n,m)\mapsto f_n(m)$ is recursive.
\end{definition}
This concept is best illustrated with a \textit{non-example}:
\begin{example}
Let $A\subseteq \N$ be non-recursive. For each $n\in \N$, let $\delta_n:\N \to \N: m\mapsto \delta_{nm}$, where $\delta_{nm}$ is the Kronecker symbol. For each $n\in \N$, define $f_n:\N\to \N$ to be $(1)$ $\delta_n$ if $n\in A$ and $(2)$ identically $0$ if $n\notin A$. One readily verifies that each individual $f_n$ is recursive. On the other hand, the sequence $(f_n)_{n\in \omega}$ is not uniformly recursive. Indeed, otherwise we could solve the membership problem for $A$, using that $n\in A\iff f_n(n)=1$.
\end{example}
Using some G\"odel numbering, we can also state a version of Definition \ref{uniformdef} for sequences of functions $(f_n)_{n\in \omega}$, where $\dom(f_n)=\text{Sent}_L$ (or $\text{cdm}(f_n)=\text{Sent}_L$). One can then define a notion of \textit{uniform} decidability for sequences of $L$-structures:
\begin{definition}
A sequence $(M_n)_{n\in \omega}$ of $L$-structures is \textit{uniformly decidable} if the sequence of functions $(\chi_{M_n})_{n\in \omega}$ is uniformly recursive, i.e. if the function $\chi:\N\times \text{Sent}_L \to \{0,1\}:(n,\phi)\mapsto \chi_{M_n}(\phi)$ is recursive.
\end{definition}
\begin{rem}
If the sequence $\chi:\N\times \text{Sent}_L \to \{0,1\}:(n,\phi)\mapsto \chi_{M_n}(\phi)$ is recursive when restricted to \textit{existential} sentences, we naturally say that the sequence $(M_n)_{n\in \omega}$ is \textit{uniformly existentially decidable}. Other syntactic variants may be defined analogously. 
\end{rem}
\subsection{Interpretability} \label{intersec}
Our formalism follows closely \S 5.3 \cite{Hod}, where details and proofs may be found.
\subsubsection{Interpretations}
Given a language $L$, an \textit{unnested} atomic $L$-formula is one of the form $x=y$ or $x=c$ or $F(\overline{x})=y$ or $R(\overline{x})$, where $x,y$ are variables, $c$ is a constant symbol, $\overline{x}$ is a tuple of variables, $F$ is a function symbol and $R$ is a relation symbol of the language $L$.
\begin{definition} \label{interdef}
An $n$-dimensional interpretation of an $L$-structure $M$ in the $L'$-structure $N$ is a triple $\Gamma= (\partial_{\Gamma},\phi\mapsto \phi_{\Gamma}, f_{\Gamma})$ consisting of:
\begin{enumerate}
\item An $L'$-formula $\partial_{\Gamma}(x_1,...,x_n)$.
\item A map $\phi\mapsto \phi_{\Gamma}$, that takes an unnested atomic $L$-formula $\phi(x_1,...,x_m)$ and sends it to an $L'$-formula $\phi_{\Gamma}(\overline{y}_1,...,\overline{y}_m)$, where each $\overline{y}_i$ is an $n$-tuple of variables.

\item A surjective map $f_{\Gamma}:\partial_{\Gamma}(N^n)\twoheadrightarrow M$.
\end{enumerate}
such that for all unnested atomic $L$-formulas $\phi(x_1,...,x_m)$ and all $\overline{a}_i\in \partial_{\Gamma}(N^n)$, we have 
$$M\models \phi(f_{\Gamma} (\overline{a}_1), ...,f_{\Gamma} (\overline{a}_m)) \iff N \models  \phi_{\Gamma}(\overline{a}_1,...,\overline{a}_m)  $$
\end{definition}
An \textit{interpretation} of an $L$-structure $M$ in the $L'$-structure $N$ is an $n$-dimensional interpretation $\Gamma$, for some $n\in \N$. In that case, we also say  that $M$ is \textit{interpretable} in $N$. The formulas $\partial_{\Gamma}$ and $\phi_{\Gamma}$ (for all unnested atomic $\phi$) are the \textit{defining formulas} of $\Gamma$. 

Interpretability is a \textit{transitive} relation on the class of structures, i.e. if the $L$-structure $M$ is interpretable in the $L'$-structure $N$ and $N$ is interpretable in the $L''$-structure $P$, then there exists a \textit{composite} interpretation of $M$ in $P$ (Exercise 2, pg. 218 \cite{Hod}).

If $N$ is an $L'$-structure and $\overline{a}=(a_1,...,a_m)\in N^m$, we write $(N,\overline{a})$ for the expansion of $N$ in the language $L(\overline{c})$, which is $L$ together with an $m$-tuple of constant symbols $(c_1,...,c_m)$ with $c_i^{(N,\overline{a})}=a_i$. If $M$ is interpretable in $(N,\overline{a})$, for some $\overline{a} \in N^m$, we say that $M$ is interpretable in $N$ \textit{with parameters}. 

\bp [Reduction Theorem 5.3.2 \cite{Hod}]\label{prophod}
Let $\Gamma$ be an $n$-dimensional interpretation of an $L$-structure $M$ in the $L'$-structure $N$. There exists a map $\phi\mapsto \phi_{\Gamma}$, extending the map of Definition \ref{interdef}$(2)$, such that for every $L$-formula $\phi(x_1,...,x_m)$ and all $\overline{a}_i\in \partial_{\Gamma}(N^n)$, we have that
$$M\models \phi(f_{\Gamma} (\overline{a}_1), ...,f_{\Gamma} (\overline{a}_m)) \iff N \models  \phi_{\Gamma}(\overline{a}_1,...,\overline{a}_m)  $$
\ep 
\begin{proof}
We describe how $\phi \mapsto \phi_{\Gamma}$ is built, for completeness (omitting details). By Corollary 2.6.2 \cite{Hod}, every $L$-formula is equivalent to one in which all atomic subformulas are unnested. One can then construct $\phi \mapsto \phi_{\Gamma}$ by induction on the complexity of formulas. The base case is handled by Definition \ref{interdef}$(2)$. This definition extends inductively according to the following rules:
\begin{enumerate}
\item $(\lnot \phi)_{\Gamma}=\lnot(\phi)_{\Gamma}$.
\item $(\bigwedge_{i=1}^n \phi_i)_{\Gamma}=\bigwedge (\phi_i)_{\Gamma}$.
\item $(\forall \phi)_{\Gamma}= \forall x_1,...,x_n (\partial_{\Gamma} (x_1,...,x_n) \rightarrow \phi_{\Gamma})$
\item $(\exists \phi)_{\Gamma}= \exists x_1,...,x_n (\partial_{\Gamma} (x_1,...,x_n) \land \phi_{\Gamma})$
\end{enumerate}
The resulting map satisfies the desired conditions of the Proposition.
\end{proof}
\begin{definition}
The map $\text{Form}_L\to \text{Form}_{L'}:\phi\mapsto \phi_{\Gamma}$ constructed in the proof of Proposition \ref{prophod} is called the \textit{reduction} map of the interpretation $\Gamma$.
\end{definition}

\subsubsection{Complexity of interpretations}
The complexity of the defining formulas of an interpretation defines a measure of complexity of the interpretation itself:

\begin{definition} [\S 5.4$(a)$ \cite{Hod}]
An interpretation $\Gamma$ of an $L$-structure $M$ in an $L'$-structure $N$ is quantifier-free if the defining formulas of $\Gamma$ are quantifier-free. Other syntactic variants are defined analogously (e.g., existential interpretation).
\end{definition} 
\begin{rem} \label{existinterrem}
$(a)$ The reduction map of a positive existential interpretation sends positive existential formulas to positive existential formulas.\\
$(b)$ The reduction map of an existential interpretation sends \textit{positive existential} formulas to existential formulas but does \textit{not} necessarily send existential formulas to existential formulas. 
\end{rem}

\bl \label{transcomplexity}
If the $L$-structure $M$ is $\exists^+$-interpretable in the $L'$-structure $N$ and $N$ is $\exists^+$-interpretable in the $L''$-structure $P$, then the composite interpretation of $M$ in $P$ is also an $\exists^+$-interpretation.
\el 
\begin{proof}
Clear.
\end{proof}

\subsubsection{Recursive interpretations}

\begin{definition} [Remark 4, pg. 215 \cite{Hod} ]
Suppose $L$ is a recursive language. Let $\Gamma$ be an interpretation of an $L$-structure $M$ in the $L'$-structure $N$. We say that the interpretation $\Gamma$ is \textit{recursive} if the the map $\phi \mapsto \phi_{\Gamma}$ on unnested atomic formulas is recursive.
\end{definition}

\begin{rem} [Remark 4, pg. 215 \cite{Hod}]
If $\Gamma$ is a recursive interpretation of an $L$-structure $M$ in the $L'$-structure $N$, then the reduction map of $\Gamma$ is also recursive.
\end{rem}
\subsubsection{Uniformly recursive interpretations} \label{uniformrec}

\begin{definition}
Suppose $L$ and $L'$ are languages. Let $(M_n)_{n\in \omega}$ be a sequence of $L$-structures and $(N_n)_{n\in \omega}$ be a sequence of $L'$-structure. For each $n\in \N$, let $\Gamma_n$ be an interpretation of $M_n$ in $N_n$. We say that the sequence of interpretations $(\Gamma_n)_{n\in \omega}$ is \textit{uniformly recursive} if the sequence of reduction maps $(\phi \mapsto \phi_{\Gamma_n})_{n\in \omega}$ on unnested atomic formulas is uniformly recursive, i.e. if the map $(n,\phi)\mapsto \phi_{\Gamma_n}$ is recursive.
\end{definition}
If an $L$-structure $M$ is interpretable in the $L'$-structure $N$ and the latter is decidable, then so is the former. It is not hard to prove the following uniform version:
\bp \label{uniformdecprop}
Suppose $L$ is a recursive language, $(M_n)_{n\in \omega}$ a sequence of $L$-structures and $N$ is an $L'$-structure. Suppose $N$ is decidable, $\Gamma_n$ is an interpretation of $M_n$ in $N$ and the sequence of interpretations $(\Gamma_n)_{n\in \omega}$ is uniformly recursive. Then the sequence $(M_n)_{n\in \omega}$ is  uniformly decidable.
\ep 
\begin{proof}
Rephrasing Proposition \ref{prophod} for sentences, yields $\chi_{M_n} (\phi)=\chi_N(\phi_{\Gamma_n})$ for every $\phi \in \text{Sent}_L$. It follows that the map $(n,\phi)\mapsto  \chi_{M_n} (\phi)$ is equal to the map $(n,\phi)\mapsto \phi_{\Gamma_n}\mapsto \chi_N(\phi_{\Gamma_n})$ and the latter is recursive as a composition of recursive functions.
\end{proof}
If the interpretation of $M$ in $N$ is recursive and so is the interpretation of $N$ in $P$, then the composite interpretation of $M$ in $P$ is recursive as well. Indeed, recursive functions are closed under composition. One also has a uniform version: 
\bl \label{transitivity}
Let $(M_n)_{n\in \omega}$ be a sequence of $L$-structures, $(N_n)_{n\in \omega}$ be a sequence of $L'$-structures and $(P_n)_{n\in \omega}$ be a sequence of $L''$-structures. For each $n\in \N$, let $\Gamma_n$ be an interpretation of $M_n$ in $N_n$ and $\Delta_n$ be an interpretation of $N_n$ in $P_n$ and suppose that the sequences of interpretations $(\Gamma_n)_{n\in \omega}$ and $(\Delta_n)_{n\in \omega}$ are uniformly recursive. Let $E_n$ be the composite interpretation of $M_n$ in $P_n$. Then the sequence of interpretations $(P_n)_{n\in \omega}$ is uniformly recursive.
\el
\begin{proof}
Clear.
\end{proof}

\section{Ax-Kochen/Ershov in mixed characteristic} \label{sec1}

\subsection{A result by van den Dries }
\subsubsection{Introduction}
We start with an Ax-Kochen/Ershov style result due to van den Dries (unpublished), which is briefly discussed on pg. 144 in \cite{vdd}. We shall sketch the proof (due to van den Dries), which does not seem to appear anywhere in the published literature. Some references, which use a similar \textit{coarsening} argument, include pg. 2 \cite{vddwitt}, the proof of Corollary 12.3 \cite{AJ} and the proof of Proposition 9.6 \cite{Scanlon}. For background material on coarsenings of valuations, see \S 7.4 \cite{vdd}.
\subsubsection{Inverse systems}
The formalism of multi-sorted structures, used in this section, is spelled out in \S 3 \cite{Scanlon}. The proof of van den Dries' Theorem \ref{thO/p^nO} requires a technical lemma for inverse systems, which we now discuss.

Let $\mathcal{R}=(R_n)_{n\in \omega}$ be a sequence of rings, viewed as a multi-sorted structure with sorts $(\mathbf{R}_n)_{n\in \omega}$, each equipped with the language of rings $L_{\text{rings}}$, and for each $n\in \N$ we have a map $f_n:\mathbf{R}_{n+1}\to \mathbf{R}_n$. Let $T$ be the theory that requires of $\mathcal{R}=(R_n)_{n\in \omega}$ that the map $f_n: R_{n+1}\to R_n$ be a surjective ring homomorphism with $\text{Ker}(f_n)=p^n R_{n+1}$ 
, i.e. $R_n\cong R_{n+1}/p^nR_{n+1}$. If $\mathcal{R}=(R_n)_{n\in \omega}$ and $\mathcal{S}=(S_n)_{n\in \omega}$ are two models of $T$, they are isomorphic precisely when there is a \textit{compatible} system $(\phi_n)_{n\in \omega}$ of isomorphisms $\phi_n:R_n\xrightarrow {\cong} S_n$, i.e. such that the diagram commutes

\[
\begin{tikzcd}
R_{n+1} \arrow[r, "\phi_{n+1}"]  \arrow["f_n",d] 
& S_{n+1} \arrow[d, "g_n"] \\
R_n \arrow[r, "\phi_n"]
& [blue] S_n
\end{tikzcd}
\]
for each $n\in \N$. Compatibility of the $\phi_n$'s is essential as there are examples where $R_n\cong S_n$ for each $n\in \N$ but $\mathcal{R}\not \cong \mathcal{S}$ (see e.g., Remark \ref{kedlayatemkin}). Somewhat surprisingly, compatibility comes for free in a saturated setting: 
\bl \label{invlemma}
Assume \text{CH}. Let $\mathcal{R}=(R_n)_{n\in \omega}$ and $\mathcal{S}=(S_n)_{n\in \omega}$ be two models of $T$. Suppose that for each $n\in \N$, we have that $R_n\cong S_n$ and the rings $R_n$, $S_n$ are saturated with $|R_n|=|S_n|\leq \aleph_1$. Then $\mathcal{R}\cong \mathcal{S}$.
\el
\begin{proof}
Let $U$ be a non-principal ultrafilter on $\mathcal{P}(\N)$ and consider the ultraproducts $R_{U}=\prod_{n\in \omega} R_n/U$ and $S_{U}=\prod_{n\in \omega} S_n/U$. \\
\textbf{Claim 1:} $R_{U}\cong S_{U}$.
\begin{proof}
Since $R_n\equiv S_n$ for each $n\in \N$, we get that $R_{U}\equiv S_{U} $ by \L o\'s' s Theorem. For each $n\in \N$, let $\mathcal{F}_m=\{n\in \N: |R_n|\geq m\}$. If there exists $m\in \N$ with $\mathcal{F}_m\notin U$, then $R_{U}$ and $S_U$ are both finite and thus $R_U\cong S_U$. If on the other hand $\mathcal{F}_m \in U$ for all $m\in \N$, then $R_U$ and $S_U$ are both infinite. Since $|R_n|\leq \aleph_1$, we get that $|R_U| \leq \aleph_1^{\aleph_0}=2^{\aleph_0^2}=2^{\aleph_0}=\aleph_1$, using the continuum hypothesis (similarly $|S_U|\leq \aleph_1$). Moreover, the ultraproducts $R_U$ and $S_U$ are $\aleph_1$-saturated (Exercise 4.5.37 \cite{Mark}) and thus saturated of size $\aleph_1$. Since $R_{U}$ and $S_{U}$ are elementary equivalent and both saturated of size $\aleph_1$, we conclude that $R_{U}\cong S_{U}$ (Theorem 4.3.20 \cite{Mark}). 
\qedhere $_{\textit{Claim 1}}$ \end{proof}
Next we prove:\\
\textbf{Claim 2:} For each $n\in \N$, we have $R_{U}/p^n R_{U}\cong R_n$.
\begin{proof}
Fix $n\in \N$. Since $R_m/p^nR_m\cong R_n$ for $m>n$, we get that $R_{U}/p^n R_{U}\equiv R_n$ in $L_{\text{rings}}$ by \L o\'s Theorem. If $R_n$ is finite, then $R_{U}/p^n R_{U}\cong R_n$ and we are done. Otherwise, we will have that $R_n$ is saturated of size $\aleph_1$. The same is true for $R_{U}/p^nR_{U}$, being interpretable in the structure $R_U$, which is saturated of size $\aleph_1$ by \text{CH} (see the proof of Claim 1). We conclude that $R_{U}/p^n R_{U}\cong R_n$ (Theorem 4.3.20 \cite{Mark}). 
\qedhere $_{\textit{Claim 2}}$ \end{proof}
Similarly, for each $n\in \N$, we have $S_{U}/p^n S_{U}\cong S_n$.
By Claim 1, we obtain $\phi: R_{U} \xrightarrow{\cong} S_{U}$. Note that $\phi (p^n R_U)=p^n S_U$, for each $n\in \N$. By Claim 2, this gives rise to a \textit{compatible} system $(\phi_n)_{n\in \omega}$ of isomorphisms $\phi_n: R_n \xrightarrow{\cong} S_n$, which yields $\mathcal{R}\xrightarrow{\cong} \mathcal{S}$.
\end{proof}
\subsubsection{Statement and proof}
The following result will be of fundamental importance for the rest of the paper:
\bt [van den Dries]{\label{thO/p^nO}} 
Let $(K,v),(K',v')$ be two henselian valued fields of mixed characteristic. Then $(K,v)\equiv (K',v')$ in $L_{\text{val}}$ if and only if $\mathcal{O}_v/p^n \Oo_v\equiv \Oo_{v'}/p^n\Oo_{v'}$ in $L_{\text{rings}}$ for all $n\in \N$ and $(\Gamma_v,vp)\equiv (\Gamma_{v'},v'p)$ in $L_{oag}$ together with a constant for $vp$.
\et 
\begin{proof}
"$\Rightarrow$": 
Clear.\\
"$\Leftarrow$": As the statement at hand is \textit{absolute}, we may assume the continuum hypothesis (see \S 8 \cite{Scanlon} or pg. 122 \cite{vdd}). We may therefore assume that both $(K,v)$ and $(K',v')$ are saturated of size $\aleph_1$ (Corolllary 4.3.13 \cite{Mark}). By our assumption, we have an isomorphism of ordered abelian groups $ (\Gamma_v,vp)\cong (\Gamma_{v'},v'p)$ and a ring isomorphism $\Oo_{v'}/p^n\Oo_{v'} \cong \mathcal{O}_v/p^n\Oo_v$ for each $n\in \N$. We shall argue that $(K,v)\cong (K',v')$.

Consider the finest coarsening $w$ of $v$ for which the associated residue field $k_w$ has characteristic $0$. The corresponding valuation ring is $\mathcal{O}_w=\mathcal{O}_v[\frac{1}{p}]$ and the corresponding value group is $\Gamma_w=\Gamma_v/\text{Conv}(\Z vp)$, where $\text{Conv}(\Z vp)$ is the convex hull of $\Z vp$ in $\Gamma_v$. Let $\bar v$ be the induced valuation from $v$ on the residue field $k_w$. We then have that $\mathcal{O}_{\bar v}=\mathcal{O}_v/\bigcap_{n\in \omega} p^n\mathcal{O}_v$. We also consider the analogous objects for $K'$.\\ 
\textbf{Claim 1:} We have a ring isomorphism $\mathcal{O}_{\bar v} \cong \varprojlim \mathcal{O}_v/p^n\mathcal{O}_v $.
\begin{proof}
Consider the ring homomorphism $f:\mathcal{O}_{\bar v} \to \varprojlim \mathcal{O}_v/p^n\mathcal{O}_v :x+\bigcap_{n\in \omega} p^n\mathcal{O}_v\mapsto (x+p^n\Oo_v)_{n\in \omega}$, which is clearly injective. We shall argue that it is also surjective. For a given $(x_n+p^n \Oo_v)_{n\in \omega} \in \varprojlim \mathcal{O}_v/p^n\mathcal{O}_v$, we may find $x\in \Oo_v$ with $x\equiv x_n \mod p^n \Oo_v$, using that $\Oo_v$ is $\aleph_1$-saturated. It follows that $f(x)=(x_n+p^n \Oo_v)_{n\in \omega} $.
\qedhere $_{\textit{Claim 1}}$ \end{proof}
Similarly, one obtains an isomorphism $\mathcal{O}_{\overline{ v'}} \cong \varprojlim \mathcal{O}_{v'}/p^n\mathcal{O}_{v'} $.\\
\textbf{Claim 2:} We have an isomorphism of valued fields $\phi:(k_w,\overline{v})\cong (k_{w'},\overline{v'})$.
\begin{proof}
Let $\mathcal{R}=(\Oo_v/p^n \Oo_v)_{n\in \omega}$ and $\mathcal{S}=(\Oo_{v'}/p^n \Oo_{v'})_{n\in \omega}$. Then $\mathcal{R}\cong \mathcal{S}$ by Lemma \ref{invlemma}. This yields $\varprojlim \mathcal{O}_v/p^n\mathcal{O}_v \cong \varprojlim \mathcal{O}_{v'}/p^n\mathcal{O}_{v'}$ and thus $\mathcal{O}_{\bar v} \cong \mathcal{O}_{\overline {v'}}$ by Claim 1.
\qedhere $_{\textit{Claim 2}}$ \end{proof}

We also have that $\Gamma_w\cong \Gamma_{w'}$, since the isomorphism $ (\Gamma_v,vp)\cong (\Gamma_{v'},v'p)$ descends to the quotients $\Gamma_{v}/\text{Conv}(\Z vp)\cong \Gamma_{v'}/\text{Conv}(\Z v'p)$. By Lemma 7.13 \cite{vdd}, the coarsened valued fields $(K,w)$ and $(K',w')$ are henselian too. By the Ax-Kochen/Ershov principle in pure characteristic $0$ (see e.g., Corollary 5.22 \cite{vdd}), we get that $(K,w)\equiv (K',w')$ in $L_{\text{val}}$.

Passing once again to elementary extensions, in a suitable language that includes unary predicates for both $\Oo_v$ and $\Oo_w$, we may even assume that $(K,w)\cong (K',w')$ to begin with. By stable embeddedness of residue fields for henselian valued fields of pure characteristic $0$ (see Corollary 5.25 \cite{vdd}), there is even an isomorphism $\Phi:(K,w)\xrightarrow{\cong} (K',w')$ inducing $\phi: (k_w, \overline{v})\xrightarrow{\cong} (k_{w'},\overline{v'})$. \\
\textbf{Claim 3:} The map $\Phi$ is an isomorphism of the valued fields $(K,v)$ and $(K',v')$.
\begin{proof}
Given $x\in K$, we need to show that $vx\geq 0 \iff v' (\Phi(x))\geq 0$. If $vx\geq 0$, then either $(i)$ $wx>0$ or $(ii)$ $wx=0$ and $\overline{v} \bar x\geq 0$, where $\bar x\in k_v$ is the image of $x$ via $\text{res}_w:\Oo_{w}\to k_w$. In the first case, we get that $w' (\Phi(x))>0$ as $\Phi:(K,w)\to (K',w')$ is a valued field homomorphism and therefore $v' (\Phi(x))>0$ as $\mathfrak{m}_{w'}\subset \mathfrak{m}_{v'}$. Suppose now that $wx=0$ and $\overline{v} \bar x \geq 0$. Then we also get that $w'(\Phi(x))=0$ and $\overline{v'} (\phi(\overline{x}))\geq 0$ as $\phi:(k_w,\overline{v})\to  (k_{w'},\overline{v'})$ is a valued field homomorphism. Since $\Phi$ induces $\phi$, we get that $\overline{v'}(\overline{\Phi(x)})=\overline{v'} (\phi(\overline{x}))\geq 0$ and conclude that $v' (\Phi(x))\geq 0$. 
\qedhere $_{\textit{Claim 3}}$ \end{proof}
Claim 3 finishes the proof.
\end{proof}

\subsection{Existential AKE in mixed characteristic  } 
In this section we prove an existential version of Theorem \ref{thO/p^nO}. We first review some known AKE results in the equal characteristic setting.
\subsubsection{Comparison with the equal characteristic case}
For equal characteristic henselian valued fields one has the following simple Ax-Kochen/Ershov principles due to Anscombe-Fehm:
\bt [Corollary 1.2 \cite{AnscombeFehm}] \label{ansfehm1}
Let $(K,v), (K',v')$ be two equal characteristic non-trivially valued henselian fields. Then $(K,v)\equiv_{\exists} (K',v')$ in $L_{\text{val}}$ if and only if $k\equiv_{\exists} k'$ in $L_{\text{rings}}$.
\et  
\bt  [Corollary 7.5 \cite{AnscombeFehm}] \label{AnsFehm}
Let $(K,v)$ be an equal characteristic henselian valued field. Then $Th_{\exists}(K,v)$ is decidable in $L_{\text{val}}$ if and only if $Th_{\exists}(k)$ is decidable in $L_{\text{rings}}$.
\et 
\begin{rem} \label{akevsansfehm}
In residue characteristic $0$, Theorems \ref{ansfehm1} and \ref{AnsFehm} were essentially known by work of Ax-Kochen/Ershov prior to the work of Anscombe-Fehm (see Remark 7.3 \cite{AnscombeFehm}). 
\end{rem}

\subsubsection{Existential AKE in mixed characteristic}
In mixed characteristic, one can easily construct counterexamples of Theorems \ref{ansfehm1} and \ref{AnsFehm} (see Remark 7.6 \cite{AnscombeFehm}). It is then natural to ask what an existential AKE principle in mixed characteristic would look like. This will be Theorem \ref{exthO/p^nO} below, whose proof follows closely the proof of Theorem \ref{thO/p^nO}.

We write $\mathfrak{m}_n$ for the maximal ideal of $\Oo_v/p^n \Oo_v$ and $(\Oo_v/p^n \Oo_v,\mathfrak{m}_n)$ for the local ring, viewed as an $L_{\text{lcr}}$-structure (see notation). 
\bt \label{exthO/p^nO}
Let $(K,v),(K',v')$ be two henselian valued fields of mixed characteristic. Then the following are equivalent: 
\begin{enumerate}
\item $(K,v)\equiv_{\exists} (K',v')$ in $L_{\text{val}}$.

\item $\mathcal{O}_v/p^n\Oo_v\equiv_{\exists} \Oo_{v'}/p^n \Oo_{v'}$ in $L_{\text{rings}}$ for all $n\in \N$.

\item $(\mathcal{O}_v/p^n \Oo_v, \mathfrak{m}_n) \equiv_{\exists^+} (\Oo_{v'}/p^n \Oo_{v'},\mathfrak{m}_n')$ in $L_{\text{lcr}}$ for all $n\in \N$.
\end{enumerate}
\et
\begin{proof} $(1)\Rightarrow (2),(3)$: Clear.\\
$(2)\Rightarrow (1)$: By symmetry, it will suffice to show that $(K,v)\models \text{Th}_{\exists}(K',v')$. We may further assume $(K',v')$ is countable by downward L\"owenheim-Skolem and that $(K,v)$ is $\aleph_1$-saturated.

We again consider the valuations $w,\overline{v}$ (resp. $w',\overline{v'}$) that were introduced in the proof of Theorem \ref{thO/p^nO}. By our assumption, we have 
for each $n\in \N$ an embedding of rings $\Oo_{v'}/p^n\Oo_{v'}\hookrightarrow \mathcal{O}_v/p^n\Oo_v$.\\ 
\textbf{Claim:} There is an injective ring embedding $\phi: \mathcal{O}_{\bar v'}\hookrightarrow \mathcal{O}_{\bar v}$.
\begin{proof}
Let us fix an enumeration of $\mathcal{O}_{v'}$, say $\mathcal{O}_{v'}=(a_i)_{i\in \N}$. Consider the set of formulas in countably many variables $x=(x_n)_{n\in \omega}$ of the form
$$ \Sigma (x)=\{x_i\diamond x_j= x_k (p^n), x_m  \Box x_{\rho}(p^n): a_i\in \mathcal{O}_{v'}, a_i\diamond a_j=a_k (p^n), a_m\Box a_{\rho}(p^n)\}$$
where $\diamond$ is either $+$ or $\cdot$ and $\Box$ is either $=$ or $\neq$. Since for each $n\in \N$ we have an embedding of rings $\mathcal{O}_{v'}/p^n\mathcal{O}_{v'} \hookrightarrow \mathcal{O}_v/p^n\mathcal{O}_v$, we get that $\Sigma(x)$ is finitely satisfiable. The ring $\mathcal{O}_v$ is $\aleph_1$-saturated and we thus have $b=(b_n)_{n\in \omega}$ with $b\models \Sigma(x)$. Using that $\mathcal{O}_{\bar v}=\mathcal{O}_v/\bigcap_{n\in \omega} p^n\mathcal{O}_v$ (resp. $\mathcal{O}_{\overline{v'}}=\mathcal{O}_{v'}/\bigcap_{n\in \omega} p^n\mathcal{O}_{v'}$), one readily checks that the map $\mathcal{O}_{v'}\to\mathcal{O}_{v}:a_i\mapsto b_i $ descends to a ring embedding $\mathcal{O}_{\bar v'}\hookrightarrow \mathcal{O}_{\bar v}$.
\qedhere $_{\textit{Claim}}$ \end{proof}
The Claim provides us with a valued field embedding $\phi:(k_{w'},\bar v')\hookrightarrow (k_w,\bar v)$. 
By the existential Ax-Kochen/Ershov principle in pure characteristic $0$ (see Theorem \ref{ansfehm1} and Remark \ref{akevsansfehm}), we get that $(K,w)\models \text{Th}_{\exists}(K',w')$. Replacing $K$ with an $\aleph_1$-saturated extension in a suitable language that includes unary predicates for both $\Oo_v$ and $\Oo_w$, we will also have an embedding $(K',w')\hookrightarrow (K,w)$.

By the relative embedding property for equal characteristic $0$ henselian valued fields (see Theorem 7.1 in \cite{Kuhl} for a more general statement), we can even find $\Phi:(K',w')\hookrightarrow (K,w) $ that induces $\phi:(k_{w'},\bar v')\hookrightarrow (k_w,\bar v)$. Finally, we get that the map $\Phi: (K',v')\hookrightarrow (K,v)$ is an embedding of valued fields, as in the proof of Claim 3, Theorem \ref{thO/p^nO}.\\
$(3)\Rightarrow (2)$: For $f(x_1,...,x_m)\in \Z[x_1,...,x_m]$ and $(a_1,...,a_m) \in \Oo_v^m$, note that $f(a_1,...,a_m) \neq 0 \mod p^n \Oo_v$ if and only if there exists $y\in \mathfrak{m}_v$ such that $f(a_1 ,...,a_m) \cdot y=p^n \mod p^{n+1} \Oo_v$ (similarly for $\Oo_{v'}$). 
Consequently, for each $n\in \N$, we see that if $(\Oo_v/p^{n+1} \Oo_v, \mathfrak{m}_{n+1})\equiv_{\exists^+} (\Oo_{v'}/p^{n+1} \Oo_{v'}, \mathfrak{m}'_{n+1})$ in $L_{\text{lcr}}$, then $\Oo_v/p^n\Oo_v\equiv_{\exists} \Oo_{v'}/p^n \Oo_{v'}$ in $L_{\text{rings}}$. 
\end{proof}
\subsection{Decidability} \label{uniformitydecidab}
We now harvest the consequences of Theorem \ref{thO/p^nO} and \ref{exthO/p^nO} in relation to decidability. Since the countable union of recursive sets is not guaranteed to be recursive, we need to ask not only that each individual $\Oo_K/(p^n)$ be decidable in $L_{\text{rings}}$ but also that the sequence $(\Oo_K/(p^n))_{n\in \omega}$ be \textit{uniformly decidable} in $L_{\text{rings}}$:
\bc \label{vddcor}
Let $(K,v)$ be a henselian valued field of mixed characteristic. Then the following are equivalent: 
\begin{enumerate}
\item The valued field $(K,v)$ is decidable in $L_{\text{val}}$.

\item The sequence $(\Oo_K/(p^n))_{n\in \omega}$ is uniformly decidable in $L_{\text{rings}}$ and $(\Gamma_v,vp)$ is decidable in $L_{\text{oag}}$ with a constant for $vp$.
\end{enumerate}

\ec
\begin{proof}
$(1)\Rightarrow (2)$: Clear.\\
$(2) \Rightarrow (1)$: The identification $\Gamma_v=K^{\times}/\Oo^{\times}$ furnish us with a recursive interpretation $E$ of $(\Gamma_v,vp)$ in the valued field $(K,v)$. Let also $E_n$ be the natural interpretation of $\Oo_K/(p^n)$ in the $L_{\text{val}}$-structure $(K,v)$, for each $n\in \N$. If $g_n:\text{Sent}_{L_{\text{rings}}}\to \text{Sent}_{L_{\text{val}}}$ denotes the reduction map of $E_n$, then one can see that the sequence $(g_n)_{n\in \omega}$ is uniformly recursive  (using that $E_n$ is uniform in $n\in \N$). Let 
$$\Sigma:=\text{Hen}_{(0,p)}\cup (\bigcup_{n\in \omega} \{\phi_{E_n}:\Oo_K/(p^n)\models \phi \}) \cup \{ \phi_{E}: (\Gamma_v,vp)\models \phi\}$$ 
where $\text{Hen}_{(0,p)}$ is a first-order axiom schema capturing Hensel's lemma (see e.g., pg.21 \cite{Kuhl}), together with a set of sentences capturing that the valued field has mixed characteristic $(0,p)$.\\
\textbf{Claim:} The axiomatization $\Sigma$ is r.e.
\begin{proof}
The set $\text{Hen}_{(0,p)}$ is clearly r.e. The set $\{ \phi_{E}: (\Gamma_v,vp)\models \phi\}\subseteq L_{\text{val}}$ is r.e., being the image of a recursive set via the recursive reduction map of $E$. Since recursively enumerable sets are closed under finite unions, it remains to show that $\bigcup_{n\in \omega} \{\phi_{E_n}:\Oo_K/(p^n)\models \phi \}$ is r.e. 

Let $\chi: \N\times \text{Sent}_{L_{\text{rings}}}\to \N$ be the recursive function associated to the uniformly decidable sequence $(\Oo_K/(p^n))_{n\in \omega}$. We construct the partial recursive function $\chi': \N\times \text{Sent}_{L_{\text{rings}}} \to \N\times \text{Sent}_{L_{\text{rings}}}$ which maps $(n,\phi)\mapsto (n,\phi)$ if $\chi(n,\phi)=1$ and is undefined if $\chi(n,\phi)=0$. Consider also the recursive function $g:\N\times \text{Sent}_{L_{\text{rings}}}\to \text{Sent}_{L_{\text{val}}}:(n,m)\mapsto g_n(m) $ associated to the uniformly recursive sequence $(g_n)_{n\in \omega}$. Observe that $\bigcup_{n\in \omega} \{\phi_{E_n}:\Oo_K/(p^n)\models \phi \}=\text{Im}(F)$ where $F$ is the (partial) recursive function $F= g \circ \chi'$. It follows that $\bigcup_{n\in \omega} \{\phi_{E_n}:\Oo_K/(p^n)\models \phi \}$ is r.e.
\qedhere $_{\textit{Claim}}$ \end{proof}
If $(K',v')\models \Sigma$, then $(K',v')\equiv (K,v)$ in $L_{\text{val}}$ by Theorem \ref{thO/p^nO}. We therefore get that $\Sigma$ is a complete axiomatization of $(K,v)$. 
We conclude that the $L_{\text{val}}$-theory of $(K,v)$ admits a r.e. and complete axiomatization, whence $(K,v)$ is decidable.
\end{proof}
\bc \label{vddcor2}
Let $(K,v)$ be a henselian valued field of mixed characteristic. Then the following are equivalent: 
\begin{enumerate}
\item The valued field $(K,v)$ is $\exists$-decidable in $L_{\text{val}}$.

\item  The sequence $(\Oo_K/(p^n))_{n\in \omega}$ is uniformly $\exists$-decidable in $L_{\text{rings}}$.

\item The sequence $((\Oo_K/(p^n),\mathfrak{m}_n)_{n\in \omega}$ is uniformly $\exists^+$-decidable in $L_{\text{lcr}}$.
\end{enumerate}
\ec
\begin{proof}
Similar to Corollary \ref{vddcor}, ultimately using Theorem \ref{exthO/p^nO}.
\end{proof}


\section{Perfectoid fields} \label{localapprox}
\subsection{Introduction}
\subsubsection{Motivation}
The theory of perfectoid fields (and spaces), introduced by Scholze in \cite{Scholze}, was initially designed as a means of transferring results available in positive characteristic to mixed characteristic (see \S 1 \cite{ScholzeICM}). It formalizes the earlier Krasner-Kazhdan-Deligne philosophy (due to \cite{Krasner}, \cite{Kazhdan} and \cite{Deligne}), of approximating a \textit{highly ramified} mixed characteristic field with a positive characteristic field. Within the framework of perfectoid fields, this kind of approximation becomes precise and robust with the use of the tilting functor (see \S \ref{perfsec}). 

All this is substantially different from the Ax-Kochen method (see e.g., 2.20 \cite{vdd}), which achieves a model-theoretic transfer principle \textit{asymptotically}, i.e. with the residue characteristic $p\to \infty$. The theory of perfectoid fields will allow us to transport decidability information for a fixed residue characteristic (but with high ramification), setting the stage for a different type of model-theoretic transfer principle.
\subsubsection{Definition}
\begin{definition}
A perfectoid field is a complete valued field $(K,v)$ of residue characteristic $p>0$ such that $\Gamma_v$ is a dense subgroup of $\Rr$ and the Frobenius map $\Phi : \mathcal{O}_K/(p) \to \mathcal{O}_K/(p):x\mapsto x^p$ is surjective.
\end{definition}
\begin{example}
$(a)$ The $p$-adic completions of $\Q_p(p^{1/p^{\infty}})$, $\Q_p(\zeta_{p^{\infty}})$ and $\Q_p^{ab}$ are mixed characteristic perfectoid fields.\\
$(b)$ The $t$-adic completions of $\F_p(\!(t)\!)^{1/p^{\infty}}$ and $\overline{\F}_p(\!(t)\!)^{1/p^{\infty}}$ are perfectoid fields of characteristic $p$.
\end{example}
\begin{rem}
In characteristic $p$, a perfectoid field is simply a perfect, complete non-archimedean valued field of rank $1$. 
\end{rem}

\subsection{Tilting} \label{perfsec}
\subsubsection{Introduction}
A construction, originally due to Fontaine, provides us with a tilting functor that takes any perfectoid field $K$ and transforms it into a perfectoid field $K^{\flat}$ of characteristic $p$. We shall now describe this tilting functor. For more details, see \S 3 \cite{Scholze}.
\subsubsection{Definition}
Given a perfectoid field $(K,v)$, we shall now define its tilt $(K^{\flat},v^{\flat})$. Let $\varprojlim_{x\mapsto x^p}K$ be the limit of the inverse system 
$$  ...\xrightarrow{x\mapsto x^p} K\xrightarrow{x\mapsto x^p} K \xrightarrow{x\mapsto x^p} K$$
which is identified as $\varprojlim_{x\mapsto x^p}K=\{ (x_n)_{n\in \omega}: x_{n+1}^p=x_n\}$, viewed as a multiplicative monoid via $(x_n)_{n\in \omega} \cdot (y_n)_{n\in \omega}=(x_n\cdot y_n)_{n\in \omega} $. 
Similarly, one can define the multiplicative monoid $ \varprojlim _{x\mapsto x^p} \Oo_K$.

Let $\varpi \in \Oo_K$ be such that $0<v \varpi \leq vp$ (e.g., $\varpi=p$ when $\text{char}(K)=0$ and $\varpi=0$ when $\text{char}(K)=p$) and consider the ring $\varprojlim_{\Phi} \Oo_K/(\varpi)$ which is the limit of the inverse system of rings
$$...\xrightarrow{\Phi} \Oo_K/(\varpi)\xrightarrow{\Phi} \Oo_K/(\varpi) \xrightarrow{\Phi} \Oo_K/(\varpi)$$
where $\Phi: \Oo_K/(\varpi)\to \Oo_K/(\varpi):x\mapsto x^p$ is the Frobenius homomorphism.
\bl [Lemma 3.4 $(i),(ii)$ \cite{Scholze}] \label{lemscholz1}
$(a)$ The ring $\varprojlim_{\Phi} \Oo_K/(\varpi)$ is independent of the choice of $\varpi$ and there is a multiplicative isomorphism $\varprojlim_{x\mapsto x^p} \Oo_K \xrightarrow{\cong} \varprojlim_{\Phi} \Oo_K/(\varpi)$. Moreover, we get a multiplicative morphism $\sharp:\varprojlim_{\Phi} \Oo_K/(\varpi) \to \Oo_K:x\mapsto x^{\sharp}$ such that if $x=(x_n+(\varpi))_{n\in \omega}$, then $x^{\sharp}\equiv x_0 \mod (\varpi)$.\\
$(b)$ There is an element $\varpi^{\flat} \in \varprojlim_{\Phi} \Oo_K/(\varpi)$ with $v(\varpi^{\flat})^{\sharp}=v\varpi$. 
\el
The definition of $\sharp$ goes as follows: Let $x=(x_n)_{n\in \omega}\in \varprojlim_{\Phi} \Oo_K/(\varpi)$, i.e. $x_n\in \Oo_K/(\varpi)$ and $x_{n+1}^p=x_n$. Let $\tilde{x}_n\in \Oo_K$ be an arbitrary lift of $x_n\in \Oo_K/(\varpi)$. Then the limit $\lim_{n\to \infty} \tilde{x}_n^{p^n}$ exists and is independent of the choice of the $\tilde{x}_n$'s (see Lemma 3.4$(i)$ \cite{Scholze}). We define $x^{\sharp}:=\lim_{n\to \infty} \tilde{x}_n^{p^n}$. 

We now introduce $K^{\flat}:=\varprojlim_{\Phi} \Oo_K/(\varpi) [(\varpi^{\flat})^{-1}]$. A priori this is merely a ring. It is in fact a valued field according to the following:
\bl [Lemma 3.4 $(iii)$ \cite{Scholze}]\label{lemscholz}
$(a)$ There is a morphism of multiplicative monoids $K^{\flat}\to K:x\mapsto x^{\sharp}$ (extending the one of Lemma \ref{lemscholz1}$(a)$), which induces a morphism of multiplicative monoids $K^{\flat} \xrightarrow {\cong} \varprojlim_{x\mapsto x^p}K:x\mapsto (x^{\sharp},(x^{\sharp})^{1/p},...)$. The map $v^{\flat}:K^{\flat}\to \Gamma_v\cup \{\infty\}:x\mapsto vx^{\sharp}$ is a valuation on $K^{\flat}$, which makes $(K^{\flat},v^{\flat})$ into a perfectoid field of characteristic $p$. If $\Oo_{K^{\flat}}$ is the valuation ring of $K^{\flat}$, then we have a ring isomorphism $\Oo_{K^{\flat}} \cong \varprojlim_{\Phi} \Oo_K/(\varpi)$.\\
$(b)$ We have an isomorphism of ordered abelian groups $(\Gamma_v,v\varpi)\cong (v^{\flat}K^{\flat}, v^{\flat} \varpi^{\flat})$ and a field isomorphism $k_v\cong k_{v^{\flat}}$. Moreover, we have a ring isomorphism $\Oo_K/(\varpi)\cong \Oo_{K^{\flat}}/(\varpi^{\flat})$.
\el 
\begin{rem} 
Lemma \ref{lemscholz}$(a)$ allows us to identify the multiplicative underlying monoid of $K^{\flat}$ with $\varprojlim_{x\mapsto x^p} K$. It is not hard to see that, via this identification, addition is described by $(x_n)_{n\in \omega}+(y_n)_{n\in \omega}=(z_n)_{n\in \omega}$, where $z_n=\lim_{m\to \infty} (x_{n+m}+y_{n+m})^{p^m}$.
\end{rem}
\begin{definition}
We say that the valued field $(K^{\flat},v^{\flat})$ constructed in Lemma \ref{lemscholz}$(a)$ is the tilt of the perfectoid field $(K,v)$.
\end{definition}

\begin{rem} [Lemma 3.4 $(iv)$ \cite{Scholze}] \label{itself}
If $(K,v)$ is a perfectoid field of characteristic $p$, then $(K^{\flat},v^{\flat})\cong (K,v)$.
\end{rem}

\begin{example}[see Corollary \ref{tiltindeed}]  \label{perfex}
In the examples below, $\widehat{K}$ stands for the $p$-adic (resp. $t$-adic) completion of the field $K$ depending on whether its characteristic is $0$ or $p$.\\
$(a)$ $\widehat{\Q_p(p^{1/p^{\infty}})} ^{\flat}\cong \widehat{\F_p(\!(t)\!)^{1/p^{\infty}}}$ and $t^{\sharp}= p$.\\
$(b)$ $ \widehat{\Q_p(\zeta_{p^{\infty}})}^{\flat}\cong \widehat{\F_p(\!(t)\!)^{1/p^{\infty}}}$ and $(t+1)^{\sharp}= \zeta_p$.\\
$(c)$ $\widehat{\Q_p^{ab}}^{\flat} \cong \widehat{\overline{ \F}_p(\!(t)\!)^{1/p^{\infty}}}$ and $(t+1)^{\sharp}= \zeta_p$.
\end{example}

\begin{rem}
The tilting construction makes sense for non-perfectoid fields as well. However, in the absence of infinite wild ramification, it is too lossy for it to be useful (e.g., $\Q_p^{\flat}=\F_p$).
\end{rem}
\subsection{Witt vectors} \label{Witt}
We review the basics of Witt vectors. Details and proofs can be found in \S 5,6 \cite{Ser}, \S 3 \cite{KedLiu} and \S 6 \cite{vdd}.
\subsubsection{$p$-rings}
\begin{definition}
A ring $R$ of characteristic $p$ is called perfect if the Frobenius homomorphism $\Phi:R\to R:x\mapsto x^p$ is bijective.
\end{definition}
\begin{definition}
$(a)$ A $p$-ring is a commutative ring $A$ provided with a filtration $\mathfrak{a}_1\supset \mathfrak{a}_2\supset...$ such that $\mathfrak{a}_n\mathfrak{a}_m\subset \mathfrak{a}_{n+m}$ and so that $A$ is Hausdorff and complete with respect to the topology induced by the filtration and $A/\mathfrak{a}_1$ is a perfect ring of characteristic $p$.\\
$(b)$ If in addition $\mathfrak{a}_n=p^nA$ and $p$ is not a zero-divisor, then we say that $A$ is a strict $p$-ring.
\end{definition}
If $A$ is a $p$-ring, we call the quotient ring $A/pA$ the \textit{residue ring} of the $p$-ring $A$ and write $\text{res}:A\to A/pA$ for the quotient map. A system of \textit{multiplicative representatives} (or simply a system of representatives) is a multiplicative homomorphism $f_A: A/pA\to A$ such that $f_A(\text{res}(x))=x$. A $p$-ring $A$ always has a system of representatives $f_A:A/pA\to A$ and when $A$ is strict every element $a\in A$ can be written uniquely in the form $a=\sum_{i=0}^{\infty} f_A(\alpha_i)\cdot p^i$ (pg. 37 \cite{Ser}). 
\bt [Corollary pg.39 \cite{Ser}] \label{serthm}
For every perfect ring $R$, there exists a unique strict $p$-ring, denoted by $W(R)$, with residue ring $W(R)/pW(R)\cong R$.
\et 
The ring $W(R)$ is said to be the \textit{ring of Witt vectors} over the ring $R$. The uniqueness part of Theorem \ref{serthm} follows from the next result, which we record here for later use. 
\begin{fact} [Lemma 3.3.2 \cite{KedLiu}] \label{univwitt}
Let $A$ be a strict $p$-ring and $f_A:A/pA\to A$ be a system of representatives. Let $A'$ be a $p$-adically complete ring and $\phi: A/pA \to A'/pA'$ be a ring homomorphism. Then there exists a unique ring homomorphism $g:A\to A'$ making the diagram below commute
\[
\begin{tikzcd}
A \arrow[r, "g"] \arrow[d]
& A' \arrow[d] \\
A/pA \arrow[r, "\phi"]
& [blue] A'/pA'
\end{tikzcd}
\]
where the vertical arrows are the projections modulo $p$. More precisely, there exists a unique lift $\phi: A/pA \to A'/pA'$ to a multiplicative map $\tilde{\phi}:A/pA \to A'$ and we have $g(\sum_{i=0}^{\infty} f_A(\alpha_i) \cdot p^i)= \sum_{i=0}^{\infty} \tilde{\phi}(\alpha_i) \cdot p^i$.
\end{fact} 

\subsubsection{Teichm\"uller representatives} \label{teich}
There is a system of \textit{Teichm\"uller representatives} of $R$ in $W(R)$. This is the (unique) multiplicative homomorphism $[\ ]:R\to W(R)$ with the property that $\text{res}([x])=x$ for all $x\in R$. Explicitly, for $x\in R$ and $n\in \N$, let $x_n\in R$ be such that $x_n^{p^n}=x$ and $\tilde{x}_n\in W(R)$ be an arbitrary lift of $x_n$. The sequence $(\tilde{x}_n^{p^n})_{n\in \omega}$ is a Cauchy sequence, whose limit is independent of the chosen lifts. We let $[x]:=\lim_{n\to \infty} \tilde{x}_n^{p^n}$ (see Proposition 8 pg. 35 \cite{Ser}). 

It is easy to see that any element $x\in W(R)$ can be written \textit{uniquely} in the form $ x=\sum_{n=0}^{\infty}[x_n]\cdot p^n$, for some $x_i\in R$. The vector $(x_0,x_1,...)\in R^{\omega}$ is called the \textit{Teichm\"uller vector} of $x$. 
\subsubsection{Witt vectors}
The advantage of Witt vectors over Teichm\"uller vectors, comes from the fact that the ring operations in $W(R)$ have nicer coordinatewise descriptions when using the former (see \S \ref{ringop}). Write $x\in W(R)$ in the form $ x=\sum_{n=0}^{\infty}[x_n^{p^{-n}}]\cdot p^n$, for some $x_i\in R$. The vector $(x_0,x_1,...)\in R^{\omega}$ is called the \textit{Witt vector} of $x$. 
\subsubsection{Ring operations} \label{ringop}
By the discussion above, the ring $W(R)$ can be thought of as the $p$-adic analogue of formal power series with coefficients in $R$. By identifying $x$ with its Witt vector, we see that $W(R)$ has $R^{\omega}$ as its underlying set. By Lemma 6.5 \cite{vdd}, the ring operations are given by
$$(a_0,a_1,...)+(b_0,b_1,...)=(S_0(a_0,b_0),S_1(a_0,a_1,b_0,b_1),...)$$ 
and 
$$(a_0,a_1,...)\cdot (b_0,b_1,...)=(P_0(a_0,b_0),P_1(a_0,a_1,b_0,b_1),...)$$
for suitable polynomials $S_i,P_i \in \Z[x_0,...,x_i,y_0,...,y_i]$ which are universal, in the sense that they do not depend on $R$. 
\bob \label{compobwitt}
The polynomials $S_i$ (resp. $P_i$) are \textit{computable}, i.e. the function $\N\to \Z[x_0,y_0,...]: n\mapsto S_n$ (resp. $\N\to \Z[x_0,y_0,...]: n\mapsto P_n$) is recursive.
\eob
\begin{proof}
For $n\in \N$, we introduce the $n$-th Witt polynomial $W_n(x_0,...,x_n)=x_0^{p^n}+px_1^{p^{n-1}}+...+p^n\cdot x_n \in \Z[x_0,...,x_n]$ (pg. 135 \cite{vdd}). The proof of Lemma 6.5 \cite{vdd} shows that $S_0(x_0,y_0)=x_0+y_0$ and $W_{n-1}(S_0^p,...,S_{n-1}^p)+ p^n\cdot S_{n}=W_{n-1}(S_0(x_0^p,y_0^p),...,S_{n-1}(x_0^p,...,x_{n-1}^p,y_0^p,...,y_{n-1}^p))+p^n(x_n+y_n)$, whence the polynomial $S_n$ may be computed recursively from $S_0,...,S_{n-1}$. The proof is similar for the $P_n$'s.
\end{proof}
\subsection{Truncated Witt vectors}
\subsubsection{Definition}
In this paper, we will mostly be working with truncated Witt vectors. These can be thought of as $p$-adic analogues of truncated power series, i.e. elements of the ring $R[\![t]\!]/(t^n)\cong R[t]/(t^n)$ (over some base ring $R$). More formally:
\begin{definition}
Let $R$ be a perfect ring. Given $n\in \N$, the ring of $n$-truncated Witt vectors over $R$ is defined as $W_n(R):=W(R)/p^nW(R)$.
\end{definition}

\subsubsection{Language}
For a perfect ring $R$, the pair $(W(R),R)$ (resp. $(W_n(R),R)$) is viewed as a two-sorted structure with sorts $\textbf{W}$ for the Witt ring $W(R)$ (resp. $W_n(R)$) and $\textbf{R}$ for the residue ring $R$. The sort $\textbf{W}$ is equipped with the language of rings $L_{\text{rings}}$, while the sort $\textbf{R}$ may be equipped with any $L\supseteq L_{\text{rings}}$. We also have a function symbol for the Teichm\"uller map $[ \ ]:\textbf{R}\to \textbf{W}$. For each choice of a language $L$ for the $\textbf{R}$-sort, the resulting language will be denoted by $\langle L_{\text{rings}},L \rangle$.
\subsubsection{Interpretability in $R$}

\bl \label{trunclemma}
Let $R$ be a perfect ring, viewed as an $L$-structure with $L\supseteq L_{\text{rings}}$. For each $n\in \N$, there exists a quantifier-free interpretation $\Gamma_n$ of the $\langle L_{\text{rings}},L \rangle$-structure $(W_n(R),R)$ in the $L$-structure $R$ such that the sequence of interpretations $(\Gamma_n)_{n\in \omega}$ is uniformly recursive.
\el 
\begin{proof}
By \S \ref{ringop}, for $n\in \N$ the underlying set of $W_n(R)$ can be identified with $R^{n}$, so we take $\partial_{\Gamma_n}(x_1,...,x_n)$ to be $\bigwedge_{i=1}^n x_i=x_i$ and the coordinate map $f_{\Gamma_n}: R^n \to W_n(R)$ as the identity map. The ring operations of $W_n(R)$ are given by $(a_0,...,a_{n-1})+(b_0,...,b_{n-1})=(S_0(a_0,b_0),...,S_{n-1}(a_0,b_0,...,a_{n-1},b_{n-1}))$
and $(a_0,...,a_{n-1})\cdot (b_0,...,b_{n-1})=(P_0(a_0,b_0),...,P_{n-1}(a_0,b_0,...,a_{n-1},b_{n-1}))$, for certain polynomials $S_i,P_i \in \Z[x_0,...,x_i,y_0,...,y_i]$, for $i=0,...,n-1$. We now need to describe the map $\phi\mapsto \phi_{\Gamma_n}$ on \textit{unnested} atomic $L_{\text{rings}}$-formulas: 
\begin{enumerate}
\item If $\phi(x,y,z)$ is the formula $x+y=z$ (here $x,y,z \in \textbf{W}$), we may take $\phi_{\Gamma_n}(\overline{x}, \overline{y}, \overline{z})$ to be the $L_{\text{rings}}$-formula $\bigwedge_{i=0}^{n-1} z_i=S_i(x_0,y_0,...,x_{n-1},y_{n-1})$.

\item If $\phi(x,y,z)$ is the formula $x\cdot y=z$ (here $x,y,z \in \textbf{W}$), we may take $\phi_{\Gamma_n}(\overline{x}, \overline{y}, \overline{z})$ to be the $L_{\text{rings}}$-formula $\bigwedge_{i=0}^{n-1} z_i=P_i(x_0,y_0,...,x_{n-1},y_{n-1})$.

\item If $\phi(x,y)$ is the formula $[x]=y$ (here $x\in \textbf{R}$ and $y\in \textbf{W}$), we may take $\phi_{\Gamma_n}(x, \overline{y})$ to be the formula $\bigwedge_{i=1}^{n-1} y_i=0 \land y_0=x$.
\item If $\phi(\overline{x})$ is an unnested atomic $L$-formula with variables from the sort $\textbf{R}$, then we may take $\phi_{\Gamma_n}(\overline{x}):=\phi(\overline{x})$ (here $x_i\in \textbf{R}$ for the latter formula).
\end{enumerate}
The above data define a quantifier-free interpretation $\Gamma_n$ of the $\langle L_{\text{rings}},L \rangle$-structure $(W_n(R),R)$ in the $L$-structure $R$. Moreover, since the polynomials $S_i,P_i$ are computable (Observation \ref{compobwitt}), the sequence of interpretations $(\Gamma_n)_{n\in \omega}$ is uniformly recursive. 
\end{proof}
\subsection{Untilting} \label{Untilting}
Expository notes on the material of this section may be found either in \S 5 of the Bourbaki seminar given by Morrow \cite{Morrow} or in the lecture series notes by Lurie (see Lectures 2,3 \cite{Lu}).
\subsubsection{Overview} \label{overviewuntilt}
The functor $K\mapsto K^{\flat}$ is far from being faithful, i.e. there will be several non-isomorphic mixed characteristic perfectoid fields $K$ that tilt to the same perfectoid field of characteristic $p$. For example, the $p$-adic completions of $\Q_p(p^{1/p^{\infty}})$ and $\Q_p(\zeta_{p^{\infty}})$ both tilt to the $t$-adic completion of $\F_p(\!(t)\!)^{1/p^{\infty}}$. For a perfectoid field $F$ of characteristic $p$, an \textit{untilt} of $F$ is a pair $(K,\iota)$, where $(K,v)$ is a perfectoid field and $\iota: (F,w)\stackrel{\cong} \to(K^{\flat},v^{\flat})$ is a valued field isomorphism. Fargues-Fontaine give a description of all possible untilts of $F$ in an \textit{intrinsic} fashion, i.e. in a way that uses only arithmetic from $F$ itself (see Theorem \ref{Font}). This result will be of vital importance for Theorem \ref{reldec2}. 
\subsubsection{The ring $\textbf{A}_{\text{inf}}$}
Fix any perfectoid field $(F,w)$ of characteristic $p>0$. We introduce $\textbf{A}_{\text{inf}}:=W(\mathcal{O}_F)$, called the \textit{infinitesimal} period ring, which is the ring of Witt vectors over $\Oo_F$. 
\begin{definition}
An element $\xi \in \textbf{A}_{\text{inf}}$ is said to be \textit{distinguished} if it is of the form $\xi=\sum_{n=0}^{\infty} [c_n]p^n$, with $c_0\in \mathfrak{m}_F$ and $c_1\in \mathcal{O}_F^{\times}$. 
\end{definition}
In other words, distinguished elements are those of the form $\xi=[\pi]-up$ where $w\pi>0$ and $u\in \textbf{A}_{\text{inf}}$ is a unit. 
\begin{rem} \label{twisteddist}
Let $\text{res}:\Oo_F\to \Oo_F/\mathfrak{m}_F $ be the residue map and $W(\text{res}):\textbf{A}_{\text{inf}}\to W(\Oo_F/\mathfrak{m}_F)$ be the (unique) induced ring homomorphism provided by Fact \ref{univwitt} that maps $\sum_{n=0}^{\infty} [c_n]p^n\mapsto \sum_{n=0}^{\infty} [\text{res}(c_n)]\cdot p^n$. The element $\xi \in \textbf{A}_{\text{inf}}$ is distinguished precisely when $W(\text{res})(\xi)$ is a unit multiple of $p$ in $W(\Oo_F/\mathfrak{m}_F)$.
\end{rem}
\subsubsection{Distinguished elements and untilts}
We outline how one can go from an untilt of $F$ to an ideal of $\textbf{A}_{\text{inf}}$ generated by a distinguished element and vice versa. Let $(K,\iota)$ be an untilt of $F$, i.e. we have $\iota: (F,w) \stackrel{\cong}\to (K^{\flat},v^{\flat}) $. By Lemma \ref{lemscholz1}$(a)$, we have a morphism of multiplicative monoids $\sharp: \Oo_{K^{\flat}}\to  \Oo_K$. We also write $\sharp:\Oo_{F} \to  \Oo_K$ for the composite map $\Oo_{F} \xrightarrow{\iota} \Oo_{K^{\flat}} \xrightarrow{\sharp}  \Oo_K$. While $\sharp:\Oo_{F} \to  \Oo_K$ is not a ring homomorphism (unless $K$ has characteristic $p$), it does induce a ring homomorphism $\phi: \Oo_F\to \Oo_K/(p ):x\mapsto x^{\sharp} \mod (p)$. Moreover, $\phi$ is \textit{surjective} since it descends to an isomorphism $\Oo_F/(\pi)\xrightarrow{\cong} \Oo_K/(p)$ for any $\pi \in \Oo_F$ with $w \pi=vp$. The map $\theta$ in the lemma below is important. We shall sketch the proof of the lemma for the convenience of the reader.
\bl [Lecture 3, Remarks 11-13 \cite{Lu}] \label{thetamap}
There exists a ring homomorphism $\theta: \textbf{A}_{\text{inf}}\to \Oo_K$ inducing $\phi$ above. Moreover, $\theta$ is surjective and $\theta^{-1}(\Oo_K^{\times})=\textbf{A}_{\text{inf}}^{\times}$.
\el
\begin{hproof}
Apply Fact \ref{univwitt} with $A=\textbf{A}_{\text{inf}} $, $A'= \Oo_K$ to get that $ \phi$ lifts uniquely to the ring homomorphism 
$$\theta: \textbf{A}_{\text{inf}}\to \Oo_K:\sum_{n=0}^{\infty} [c_n]\cdot p^n\mapsto \sum_{n=0}^{\infty}  c_n^{\sharp}\cdot p^n$$
We claim that $\theta$ is surjective. Recall that $\phi$ is surjective. Given $x\in \Oo_K$, we may thus find $c_0 \in \Oo_F$ such that $x=c_0^{\sharp}+x_1 \cdot p$, for some $x_1\in \Oo_K$. Similarly, we may find $c_1 \in \Oo_F$ such that $x_1=c_1^{\sharp}+x_2\cdot p$. We then get that $x=c_0^{\sharp}+c_1^{\sharp}\cdot p+x_2\cdot p^2$. Continuing this way and since $\Oo_K$ is $p$-adically complete, we may write $x=\sum_{n=0}^{\infty}  c_n^{\sharp}\cdot p^n$. To show that $\theta^{-1}(\Oo_K^{\times})=\textbf{A}_{\text{inf}}^{\times}$, observe that  
$$\sum_{n=0}^{\infty} [c_n]\cdot p^n \in \textbf{A}_{\text{inf}}^{\times} \iff c_0 \in \Oo_F^{\times}  \iff c_0^{\sharp}\in \Oo_K^{\times} \iff  \sum_{n=0}^{\infty}  c_n^{\sharp}\cdot p^n \in \Oo_K^{\times}$$
\end{hproof}
Let us examine the kernel of $\theta$. Let $\pi \in \Oo_F$ be as above (i.e., such that $w \pi=vp$) and write $\pi^{\sharp}=\overline{u}\cdot p$ for some $\overline{u}\in \Oo_K^{\times}$. By Lemma \ref{thetamap}, we may find $u \in \textbf{A}_{\text{inf}}^{\times}$ such that $\theta(u)=\overline{u}$. Note that $\xi=[\pi]-u\cdot p\in \textbf{A}_{\text{inf}}$ is a distinguished element and that $\xi \in \text{Ker}(\theta)$. In fact, the following is true:
\bp [Corollary 17 \cite{Lu}] \label{anydisting}
Let $(F,w)$ be a perfectoid field of characteristic $p$ and $(K,\iota)$ be an untilt. Let $\theta: \textbf{A}_{\text{inf}}\to \Oo_K$ be as above. Then $\text{Ker}(\theta)$ is a principal ideal generated by \textit{any} distinguished element $\xi \in \text{Ker}(\theta)$.
\ep 
Starting with an untilt $(K,\iota)$, we have thus produced an ideal $(\xi_K)\subseteq \textbf{A}_{\text{inf}}$, where $\xi_K $ is a distinguished element in $ \textbf{A}_{\text{inf}}$.\\

Conversely, starting with $(\xi)\subseteq \textbf{A}_{\text{inf}}$, with $\xi$ a distinguished element, we may produce an untilt $(K,\iota)$ of $F$ as follows. Write $\theta: \textbf{A}_{\text{inf}} \to \textbf{A}_{\text{inf}}/(\xi)$ for the quotient map. We then have:
\bl [Lecture 3, pg. 4, Claim $(a)$ \cite{Lu})] \label{claimlur}
For every $y\in \textbf{A}_{\text{inf}}/(\xi)$, there exists $x\in \Oo_F$ such that $(y)=(\theta([x]))$.
\el 
\bp [Proposition 16 \cite{Lu}] \label{defnofvaluntilt}
Let $(F,w)$ be a perfectoid field of characteristic $p$ and $\xi \in \textbf{A}_{\text{inf}}$ a distinguished element. Then the quotient ring $\textbf{A}_{\text{inf}}/(\xi)$ is the valuation ring $\Oo_K$ of a perfectoid field $(K,v)$ such that $(K^{\flat},v^{\flat})\cong (F,w)$. The valuation $v$ is such that if $y\in \textbf{A}_{\text{inf}}/(\xi)$, then $vy:=w x$ with $x\in \Oo_F$ so that $(y)=(\theta([x]))$ in $\textbf{A}_{\text{inf}}/(\xi)$. 
\ep 
The isomorphism $\iota: (F,w) \to (K^{\flat},v^{\flat})$ of Proposition \ref{defnofvaluntilt} is described as follows. If $\xi=[\pi]-up$, then $\Oo_K/(p)\cong W(\Oo_F)/(p,[\pi]-up) \cong \Oo_F/(\pi)$. Passing to inverse limits, this induces a ring isomorphism $\Oo_{K^{\flat}}\cong \Oo_F$, which in turn yields $\iota: F\stackrel{\cong }\to K^{\flat}$ by passing to fraction fields. Starting with $(\xi)\subseteq \textbf{A}_{\text{inf}}$, we have thus produced an untilt $(K,\iota)$.

\begin{definition}
Two untilts $(K,\iota)$ and $(K',\iota')$ are isomorphic when there exists a valued field isomorphism $\phi: K \stackrel{\cong} \to  K'$ inducing a commutative diagram 
\[
\begin{tikzcd}
F \arrow[dash]{d}{=}  \arrow[r, "\iota"]  
& K^{\flat}  \arrow["\phi^{\flat}",d] \\
F \arrow[r, "\iota'"]
& K'^{\flat} 
\end{tikzcd}
\]
where $\phi^{\flat}:(x,x^{1/p},...)\mapsto (\phi(x),\phi(x^{1/p}),...)$. Let $Y_F$ denote the set of characteristic $0$ untilts of $(F,w)$, up to isomorphism. 
\end{definition}
We write $0$ for the isomorphism class of the unique characteristic $p$ untilt of $(F,\iota)$, represented by $F$ itself together with the natural isomorphism $\iota:F\stackrel{\cong} \to F^{\flat}:x\to (x,x^{1/p},...)$, and set $\overline{Y}_F=Y_F\cup \{0\}$ for the set of all untilts of $(F,w)$, up to isomorphism. 
\bt[Fargues-Fontaine] \label{Font}
Let $(F,w)$ be a perfectoid field of characteristic $p$. The map $(\xi) \mapsto \text{Frac}(\Oo_F/(\xi))$ defines a bijective correspondence between the set of
ideals $(\xi)\subseteq \textbf{A}_{\text{inf}}$ generated by a distinguished element and the set $\overline{Y}_F$.
\et  
\begin{proof}
See Proposition 5.1 \cite{Morrow} or Lecture 2, Corollary 18 \cite{Lu}.
\end{proof}
\begin{rem}
Let $(K,\iota)$ be an untilt of $(F,w)$ and $(\xi)=([\pi]-up)$ be its associated ideal. Note that $(p)=(\theta([\pi]))$ in $\textbf{A}_{\text{inf}}/(\xi)$ and therefore $v p=w \pi$.
\end{rem}

\subsubsection{Tilting equivalence} \label{tiltingequivsec}
We emphasized in \S \ref{overviewuntilt} that untilting is ambiguous, in the sense that there are many ways to untilt a perfectoid field of positive characteristic. However, the ambiguity is eliminated by \textit{fixing a base} perfectoid field $K$ and its associated tilt $K^{\flat}$. This leads to an equivalence of categories of perfectoid extensions, known as the \textit{tilting equivalence}: 
\bt [Tilting equivalence] \label{tiltingequiv}
The categories of perfectoid field extensions of $K$ and perfectoid field extensions of $K^{\flat}$ are equivalent.
\et 
\begin{proof}
See Theorem 2.8 \cite{ScholzeSurvey}. Theorem 5.2 \cite{Scholze} shows a more general result about perfectoid algebras; the case of perfectoid fields follows as a special case by Lemma 5.21 \cite{Scholze}.
\end{proof}
In the discussion after Theorem 2.8 \cite{ScholzeSurvey}, Scholze explains that there are two proofs of Theorem \ref{tiltingequiv}: 

\begin{enumerate}
\item His original proof in \cite{Scholze}, using Faltings' almost mathematics (see \S 4 \cite{Scholze}).

\item An alternative proof which describes the functor $\sharp$ inverse to $\flat$ as $L\mapsto W(\Oo_L)\otimes_{W(\Oo_{K^{\flat}})} K$ (see Remark 5.19 \cite{Scholze}).
\end{enumerate} 
Let us elaborate more on the second approach, which is in the spirit of \S \ref{Untilting} and will be more suitable for us. If $(\xi)\subset W(\Oo_{K^{\flat}})$ is the ideal associated to $K$, then one computes that $W(\Oo_L)\otimes_{W(\Oo_{K^{\flat}})} \Oo_K=W(\Oo_L)\otimes_{W(\Oo_{K^{\flat}})} W(\Oo_{K^{\flat}})/(\xi)=W(\Oo_L)/(\xi)$. In other words, if $L$ is a perfectoid field extending $K^{\flat}$, then $L^{\sharp}$ is simply the untilt of $K$ whose associated ideal in $W(\Oo_{L})$ is $(\xi)$.

\subsection{Space of untilts}
In \S \ref{modeltheorproperty} we exhibit an appealing model-theoretic property of the space $Y_F$ of untilts of $(F,w)$. This will not be used in the rest of the paper. We then study in \S \ref{size} the size of the space of untilts up to elementary equivalence. We will see that the cardinality is often too big, so that one cannot possibly expect $K$ to be decidable simply relative to $K^{\flat}$.
\subsubsection{Metric structure on $Y_F$}
Fargues-Fontaine equip $\overline{Y}_F$ with a metric topology, which allows us to view the space of untilts geometrically. Suppose $x=(K_x,v_x)$ and $y=(K_y,v_y)$ are two "points" of $\overline{Y}_F$, corresponding to the ideals $(\xi_x)$ and $(\xi_y)$ respectively, provided by Theorem \ref{Font}. We choose an embedding $\Gamma_w \hookrightarrow \R$, which determines embeddings $\Gamma_{v_y} \hookrightarrow \R$ for all $y\in Y_F$ via the canonical identification $\Gamma_{v_y} \cong \Gamma_w$. One then defines $d(x,y):=|\theta_y(\xi_x)|_y$, where $\theta_y: \textbf{A}_{\text{inf}}\twoheadrightarrow \textbf{A}_{\text{inf}}/(\xi_y)=\Oo_{K_y}$ is the quotient map and as usual $|a|_y= p^{-v_y(a)}$. 
\bp [Proposition 2.3.2$(1)$ \cite{FF}]
Let $d:\overline{Y}_F\times \overline{Y}_F\to \R$ be as above. The pair $(\overline{Y}_F,d)$ is a complete ultrametric space.
\ep 
\begin{proof}
See Propositions 2.3.2, 2.3.4 \cite{FF} or Lecture 14, Propositions 6, 7 \cite{Lu}.
\end{proof}
Recall that $0\in \overline{Y}_F$ is the isomorphism class of the untilt corresponding to $(F,w)$ itself together with the natural isomorphism $\iota:F\stackrel{\cong} \to F^{\flat}:x\to (x,x^{1/p},...)$. We can define a radius function $r(y):=d(0,y)$ for $y\in \overline{Y}_F$, which allows us to think of $\overline{Y}_F$ intuitively as the unit disc with center $0$. 
\subsubsection{A model-theoretic property of $Y_F$} \label{modeltheorproperty}
We now show that limits in the punctured unit disc $Y_F$, with respect to the Fargues-Fontaine metric, agree with limits in the model-theoretic sense:
\bp \label{obtopol}
Let $(x_n)_{n\in \omega}$ be a sequence in $Y_F$ such that $x_n\stackrel{d}{\rightarrow}x$ and $x\neq 0$. Let $(K_n,v_n)$ be the untilt associated to $x_n$ and $(K,v)$ the untilt associated to $x$. We set
$$(K^*,v^*):= \prod_{n\in \omega} (K_n,v_n)/U $$
for a non-principal ultrafilter $U$ on $\mathcal{P}(\N)$. Then $(K^*,v^*)\equiv (K,v)$ in $L_{\text{val}}$.
\ep 
\begin{proof}
Let $(\xi_n)=([\pi_n]-u_np)$ be the ideal in $\textbf{A}_{\text{inf}}$ corresponding to $(K_n,v_n)$ and $(\xi)=([\pi]-up)$ be the ideal corresponding to $(K,v)$. Note that the set $\{v_n p:n\in \N \}\subseteq \Gamma_w$ is bounded from above; otherwise, there would be a subsequence $(x_{n_k})_{k\in \N}$ with $x_{n_k}\to 0$.

Fix $m\in \N$ and let $(\bar \xi_n)$ and $(\bar \xi)$ be the images of the ideals $(\xi_n)$ and $(\xi)$ in $W_m(\mathcal{O}_F)$ via $\textbf{A}_{\text{inf}}\twoheadrightarrow \textbf{A}_{\text{inf}}/(p^m)=W_m(\mathcal{O}_F)$.\\
\textbf{Claim:} We have that $(\bar \xi_n)=(\bar \xi)$, for sufficiently large $n$.
\begin{proof}
We let $\theta: \textbf{A}_{\text{inf}}\to \textbf{A}_{\text{inf}}/(\xi)=\Oo_K$ be the quotient map. Since $d(x_n,x)\to 0$, we get that $v(\theta(\xi_n))\to \infty$. We will thus have that $\xi_n\equiv p^m \cdot \alpha_n \mod (\xi)$, for some $\alpha_n \in \textbf{A}_{\text{inf}}$ and for all $n>\!>0$. Consequently, one gets that $(\bar \xi_n)\subseteq (\bar \xi)$ for $n>\!>0$. Similarly, since $v_{n}(\theta_n(\xi))\to \infty$ and $\{v_n p:n\in \N \}$ is bounded, we get that $v_{n}(\theta_n(\xi))\geq mv_np$ for $n>\!>0$. It follows that there exists $\beta_n\in  \textbf{A}_{\text{inf}}$ such that $\xi\equiv p^m \cdot \beta_n \mod (\xi_n)$, for $n>\!>0$. We conclude that $(\bar \xi_n)=(\bar \xi)$, for $n>\!>0$.
\qedhere $_{\textit{Claim}}$ \end{proof}
It follows that $\Oo_{K_n}/(p^m) \cong W_m(\Oo_F)/(\overline{\xi}_n) = W_m(\Oo_F)/(\overline{\xi}) \cong \Oo_K/(p^m)$, for sufficiently large $n$. We also get that $(\xi_n)$ and $(\xi)$ have the same image in $\textbf{A}_{\text{inf}}/(p)\cong \Oo_F$ for $n>\!>0$ and therefore $(\pi_n)=(\pi)$ for $n>\!>0$. By Lemma \ref{lemscholz}$(b)$, we get that $(\Gamma_{v_n},v_np)\cong (\Gamma_w,w \pi_n)=(\Gamma_w,w \pi)\cong (\Gamma_v,vp)$, for all sufficiently large $n$. The conclusion follows from Theorem \ref{thO/p^nO} and \L o\'s' s Theorem.
\end{proof}

\subsubsection{The space $Z_F$} \label{size}
It is natural to consider the set
$Y_F$ up to elementary equivalence. More precisely: 
\begin{definition}
Let $(F,w)$ be a perfectoid field of characteristic $p$. For $x=(K_x,\iota_x),y=(K_y,\iota_y)\in Y_F$ define the equivalence relation $x\sim y\iff (K_x,v_x)\equiv (K_y,v_y)$ in $L_{\text{val}}$. We define $Z_F:=Y_F/\sim$.
\end{definition}
Note that the definition of $Z_F$ only takes the underlying valued fields $(K,v)$ into account and not the map $\iota$. We now determine the size of $Z_F$ in a few cases in Proposition \ref{scholzprop}. For Proposition \ref{scholzprop}$(a)$, we will need the following:
\begin{fact} [Proposition 4.3 \cite{Scholze}] \label{scholzalgclosed}
Let $(K,v)$ be a perfectoid field with $K^{\flat}$ algebraically closed. Then $K$ is also algebraically closed.
\end{fact}
For Proposition \ref{scholzprop}$(b)$ we need an algebraic fact. A finite extension $L/K$ has the \textit{unique subfield property} if for every $d|[L:K]$, there is a unique subextension $F/K$ such that $[F:K]=d$. 
\begin{fact} [Theorem 2.1 \cite{Oro}] \label{orozco}
Let $K$ be a field, $n\in \N$ be such that $\text{char}(K)\nmid n$, $a\in K$ such that $X^n-a\in K[X]$ is irreducible and set $L=K(a^{1/n})$. Suppose that for every odd prime $p|n$, we have that $\zeta_p \notin L\backslash K$ and in case $4|n$ we have $\zeta_4\notin L \backslash K$. Then $L/K$ has the unique subfield property.
\end{fact} 
The construction in Proposition \ref{scholzprop}$(b)$ is an elaborated version of Scholze's answer \cite{ScholzeMO} to a closely related mathoverflow question asked by the author. 
\bp \label{scholzprop}
$(a)$ Let $(F,w)=(\widehat{\overline{\F_p(\!(t)\!)}},v_t)$, the $t$-adic completion of an algebraic closure of $\F_p(\!(t)\!)$. Then $|Z_F|=1$.\\
$(b)$ Set $(F,w)=(\widehat{\F_p(\!(t)\!)^{1/p^{\infty}}},v_t)$. Then $|Z_F|=\mathfrak{c}$.
\ep 
\begin{proof}
$(a)$ \textbf{First proof: }Let $(K,v)$ be a perfectoid field of characteristic $0$ such that $(K^{\flat},v^{\flat})\cong (\widehat{\overline{\F_p(\!(t)\!)}},v_t)$. By Fact \ref{scholzalgclosed}, we get that $K$ is algebraically closed. By Robinson's completeness of the theory $\text{ACVF}_{(0,p)}$ of algebraically closed valued fields of mixed characteristic $(0,p)$ (see Corollary 3.34 \cite{vdd}), we get that $\text{Th}(K,v)=\text{ACVF}_{(0,p)}$ and hence $|Z_F|=1$.\\
\textbf{Second proof: } Let $\Cc_p=\widehat{\overline{\Q_p}}$ be the completed algebraic closure of $\Q_p$. By Remark 2.24 \cite{FF2}, we have a ring isomorphism $\Oo_K/(p^n)\cong \Oo_{\Cc_p}/(p^n)$ for each $n\in \N$. One also has that $(\Gamma_v,vp)\equiv (\Gamma_{\Cc_p},v_p p)$ in $L_{\text{oag}}$ with a constant for $vp$, by an easy application of quantifier elimination for the theory $\text{ODAG}$ of ordered divisible abelian groups (Corollary 3.1.17 \cite{Mark}). We conclude that $(K,v)\equiv (\Cc_p,v_p)$ from Theorem \ref{thO/p^nO}. \\
$(b)$ We assume that $p>2$; we indicate the necessary changes for the case $p=2$ in the end of the proof. For each $\alpha \in 2^{\omega}$, we define an algebraic extension $K_{\alpha}$ of $\Q_p$ as follows. We write $\alpha\restriction n$ for the restriction of $\alpha$ to $n=\{0,1,...,n-1\}$ (set-theoretically $\alpha \restriction 0=0$). We now define inductively:
\begin{enumerate}
\item $K_0=\Q_p$ and $\pi_0=p$.
\item $K_{\alpha\restriction n}=K_{\alpha\restriction (n-1)}(((1+p)^{\alpha(n-1)}\cdot \pi_{\alpha\restriction (n-1)})^{1/p} )$ and $\pi_{\alpha\restriction n}=((1+p)^{\alpha(n-1)}\cdot \pi_{\alpha\restriction (n-1)})^{1/p} $. 
\end{enumerate}
Set $\overline{\alpha}_n=\sum_{k=0}^{n-1} \alpha(k)\cdot p^k$ for $n\in \N^{>0}$ and $\overline{\alpha}_0=1$ by convention. Note that $X^{p^n}-(1+p)^{\overline{\alpha}_n}p \in \Z_p[X]$ is an Eisenstein (hence irreducible) polynomial and that $K_{\alpha\restriction n}=\Q_p(((1+p)^{\overline{\alpha}_n}p)^{1/p^n})$. We let $K_{\alpha}=\bigcup_{n\in \omega} K_{\alpha \restriction n}$. Visually, the field $K_{\alpha}$ is obtained by taking the union of all fields along a certain branch (corresponding to $\alpha$) in the binary tree below
\[
\begin{tikzpicture}[level distance=1.2cm,
  level 1/.style={sibling distance=7.6cm},
  level 2/.style={sibling distance=4cm}]
  level 3/.style={sibling distance=4cm}]
  \node {$\Q_p$}
    child {node {$\Q_p(p^{1/p})$}
      child {node {$\Q_p(p^{1/p^2})$}
           child {node {\vdots}}}
      child {node {$\Q_p(((1+p)^p p)^{1/p^2})$}
           child {node {\vdots}}}
    }
    child {node {$\Q_p(((1+p)p)^{1/p}$)}
    child {node {$\Q_p(((1+p) p)^{1/p^2})$}
         child {node {\vdots}}}
      child {node {$\Q_p(((1+p)^{1+p} p)^{1/p^2})$}
       child {node {\vdots}}}
    };
\end{tikzpicture}
\]
\textbf{Claim 1:} For each $\alpha\in 2^{\omega}$, we have $\zeta_p\notin K_{\alpha}$ and $(1+p)^{1/p}\notin K_{\alpha}$.
\begin{proof}
Note that $\zeta_{p}\notin K_{\alpha}$ since $e(K_{\alpha\restriction n}/\Q_p)=p^n$ while $e(\Q_p(\zeta_p)/\Q_p)=p-1$. Suppose that $(1+p)^{1/p}\in K_{\alpha\restriction n}$, for some $n\in \N$. By Fact \ref{orozco}, we would have that $K_{\alpha \restriction 1}=\Q_p((1+p)^{1/p})$. If $\alpha(0)=0$, this would imply that $\Q_p(p^{1/p})=\Q_p((1+p)^{1/p})
$. If $\alpha(0)= 1$,  this would imply that $\Q_p(((1+p)\cdot p)^{1/p})=\Q_p((1+p)^{1/p})$. In either case, we would get that $\Q_p(p^{1/p})=\Q_p((1+p)^{1/p})$. However, one sees that $(1+p)^{1/p}\notin \Q_p(p^{1/p})$. Indeed, suppose $a \in \Q_p(p^{1/p})$ is such that $a^p=1+p$. Write $a\equiv c_0+ c_1\cdot p^{1/p}\mod p^{2/p} \Z_p[p^{1/p}]$, for some $c_0,c_1\in \{0,...,p-1\}$. We then compute $a^p\equiv c_0^p+c_1^p\cdot p+c_0^{p-1}\cdot c_1\cdot p^{1+1/p} \not \equiv 1+p \mod p^{1+2/p}$, for any choice of $c_0$ and $c_1$. We conclude that $(1+p)^{1/p}\notin K_{\alpha}$. \\
\qedhere $_{\textit{Claim 1}}$ \end{proof}
Using Claim 1, we show:\\
\textbf{Claim 2:} If $\alpha \neq \beta$, then $\widehat{K_{\alpha}}\not \equiv_{\exists^1} \widehat{K_{\beta}}$ in $L_{\text{rings}}$. 
\begin{proof}
Note that $K_{\alpha}\equiv_{\exists^1} \widehat{K_{\alpha}}$ (resp. $K_{\beta}\equiv_{\exists^1} \widehat{K_{\beta}}$) by Theorem \ref{exthO/p^nO}. It suffices to show that $K_{\alpha} \not \equiv_{\exists^1}  K_{\beta}$ for $\alpha \neq \beta$. Let $n\in \omega$ be least such that $\alpha(n)\neq \beta(n)$ and say $\alpha(n)=1$. Now if $K_{\alpha} \equiv_{\exists^1} K_{\beta}$, there would be $a,b\in K_{\alpha}$ such that $a^{p^n}=(1+p)^{\overline{\alpha}_n} p$ and $b^{p^n}=(1+p)^{\overline{\beta}_n}p$. Set $c:=\frac{a}{b}$ and note that $c^{p^n}=\frac{(1+p)^{\overline{\alpha}_n} p}{(1+p)^{\overline{\beta}_n} p}=(1+p)^{p^{n-1}}$. This implies that $(\frac{c^p}{1+p})^{p^{n-1}}=1$. We conclude that either $\zeta_p\in K_{\alpha}$ or $c^p=1+p$, both of which are impossible by Claim 1.
\qedhere $_{\textit{Claim 2}}$ \end{proof}
Next we prove:\\
\textbf{Claim 3:} For every $\alpha \in 2^{\omega}$, we have that $\widehat{K_{\alpha}}^{\flat}=\widehat{\F_p(\!(t)\!)^{1/p^{\infty}}}$.
\begin{proof}
For every $n\in \N$, we have that $K_{\alpha\restriction (n+1)}/K_{\alpha\restriction n}$ is totally ramified and $\Oo_{K_{\alpha\restriction (n+1)}}=\Oo_{K_{\alpha\restriction n}}[\pi_{\alpha\restriction (n+1)}]$ (Corollary pg.19 \cite{Ser}). It follows that $\Oo_{K_{\alpha}}=\Z_p[\{ \pi_{\alpha\restriction n}:n\in \N\}]$. We compute 
$$\Oo_{K_{\alpha}}/(p)=\Z_p[x_1,x_2,...]/(p,x_1^p-(1+p)^{\alpha(0)}p, x_2^p-(1+p)^{\alpha(1)}x_1,...)\cong $$ 
$$\cong \F_p[x_1,x_2,...]/(x_1^{p},x_2^p-x_1,...) \cong \F_p[x_1^{1/p^{\infty}}]/(x_1^{p})\stackrel{t=x_1^p}= \F_p[t^{1/p^{\infty}}]/(t)$$
We thus get that $\varprojlim_{\Phi}\Oo_{K_{\alpha}}/(p) \cong \widehat{\F_p[\![t]\!]^{1/p^{\infty}}}$ and the conclusion follows.
\qedhere $_{\textit{Claim 3}}$ \end{proof}
Claims 2,3 show that $|Z_F|\geq \mathfrak{c}$. On the other hand, the bound $|Z_F|\leq \mathfrak{c}$ is automatic, as any first-order $L_{\text{val}}$-theory can be identified with a subset of $\text{Sent}_{L_{\text{val}}}\simeq \N$. This finishes the proof in case $p>2$. For $p=2$, one first proves a variant of Claim 1, namely that $\zeta_4\notin K_{\alpha}$ and $(\pm 3)^{1/2}\notin K_{\alpha}$. For the former, one can check inductively that  $\zeta_4\notin  K_{\alpha\restriction n}$ using Kummer theory for quadratic extensions. The proof of Claim 2 then goes through. Finally, the proof of Claim 3 works verbatim for $p=2$.
\end{proof}
\begin{rem} \label{kedlayatemkin}
Set $(F,w)=(\widehat{\overline{\F_p(\!(t)\!)}},v_t)$. In Remark 2.24 \cite{FF2}, Fargues-Fontaine ask whether every untilt $(K,v)$ of $(F,w)$ is isomorphic to $(\Cc_p,v_p)$. They note in Remark 2.24 that for any untilt $(K,v)$ of $(F,w)$ one has $\Oo_K/(p^n)\cong \Oo_{\Cc_p}/(p^n)$ for each $n\in \N$. However, these isomorphisms are obtained in a non-canonical way and do not necessarily yield a valued field isomorphism between $K$ and $\Cc_p$. In fact, an example of such a $K$ with $K \not \cong \Cc_p$ was provided by Kedlaya-Temkin in Theorem 1.3 \cite{KedTem}. This should be contrasted with Proposition \ref{scholzprop}$(a)$, saying that all untilts of $(F,w)$ are elementary equivalent.
\end{rem} 

\begin{rem}
In Remark 7.6 \cite{AnscombeFehm}, the authors write:\\"\textit{At present, we do not know of an example of a mixed characteristic henselian valued field $(K, v)$ for which $\text{Th}_{\exists}(k_v)$ and $\text{Th}_{\exists}(\Gamma_v, vp)$ are decidable but $\text{Th}_{\exists}(K, v)$ is undecidable.}"

We note that such an example indeed exists. By Proposition \ref{scholzprop}$(b)$ and the fact that there are countably many Turing machines, there must exist an undecidable perfectoid (thus henselian) field $(K,v)$ with $K^{\flat}\cong \widehat{\F_p(\!(t)\!)^{1/p^{\infty}}}$. Moreover, the proof even provides us with a field $K$ whose algebraic part is undecidable. On the other hand, by Lemma \ref{lemscholz}$(b)$, we have that $k_v\cong \F_p$ and $(\Gamma_v,vp)\cong (\frac{1}{p^{\infty}}\Z, 1 )$, both of which are decidable---the latter being an easy application of Robinson-Zakon \cite{Rob}. Finally, we note that Dittmann \cite{Dittmann} has recently provided an example which is discretely valued and whose algebraic part is decidable.
\end{rem}
\begin{rem} \label{claim2t+1}
For future use in Proposition \ref{almostdec}, we record here that Claim 3 gives $t=(\pi_{\alpha\restriction 0}+(p),\pi_{\alpha \restriction 1}+(p),...)$, via the identifications $\widehat{\F_p[\![t]\!]^{1/p^{\infty}}}\cong \varprojlim_{\Phi} \Oo_{K_{\alpha}}/(p)$.
\end{rem}

\bq 
Is there a perfectoid field $(F,w)$ of characteristic $p$ with $1<|Z_F|< \mathfrak{c}$?
\eq 

\section{Relative decidability for perfectoid fields} \label{reldecsec}

\subsection{Introduction}
We shall now use the results of \S \ref{localapprox} to prove Theorem \ref{reldec} and Corollary \ref{mainCor}.
\subsubsection{Work of Rideau-Scanlon-Simon} \label{rideauscanlon}
It is my understanding that there is ongoing work by Rideau-Scanlon-Simon, which aims at giving a model-theoretic account of many of the concepts and facts that were discussed in \S \ref{localapprox}, in the context of \textit{continuous} logic. In particular, they obtain a bi-interpretability result between $\Oo_K$ and $\Oo_{K^{\flat}}$ in the sense of continous logic. It should be noted that their bi-interpretability result was conceived prior to the present paper and in fact influenced the material that is presented here.
\subsubsection{Plan of action}
Without any adjustments, the interpretation of Rideau-Scanlon-Simon does not quite yield an interpretation in the sense of ordinary first-order logic. The problem is that when one converts statements about $\Oo_K$ to statements about $\Oo_{K^{\flat}}$ via $\Oo_K\cong W(\Oo_{K^{\flat}})/(\xi)$, one ends up with infinitely many variables, coming from the Witt vector coordinates. This problem disappears by interpreting one residue ring $\Oo_K/(p^n)$ at a time, via $\Oo_K/(p^n)\cong W_n(\Oo_{K^{\flat}})/(\xi \mod (p^n))$. This approach is facilitated by Corollary \ref{vddcor}. In \S \ref{uniformitydecidab}, we emphasized that \textit{uniform} decidability of the residue rings is key. As we will see, this eventually comes down to the computability of the distinguished element $\xi$ itself.
\subsection{Computable untilts} \label{compunt}
\subsubsection{Computable Witt vectors}
In analogy with a computable real number, a computable $p$-adic integer is one for which there is an algorithm which outputs the sequence of its $p$-adic digits. More precisely, a $p$-adic integer $a \in \Z_p$ of the form $a=\sum_{n=0}^{\infty} [\alpha_n]\cdot p^n$ is said to be \textit{computable} precisely when the function $\N\to \Z/p\Z:n\mapsto \alpha_n$ is recursive. The notion of a computable $p$-adic integer extends naturally to the more general concept of a \textit{computable Witt vector}.

First recall that a \textit{computable} ring $R_0$ is one whose underlying set is (or is identified with) a \textit{recursive} subset of $\N$, via which the operations $+:R_0\times R_0\to R_0$ and $\cdot :R_0\times R_0\to R_0$ are identified with \textit{recursive} functions.


\begin{definition} \label{compwittvector}
Fix a perfect ring $R$ and let $R_0\subseteq R$ a computable subring. Consider the ring $W(R)$, the ring of Witt vectors over $R$. An element $\xi\in W(R)$ with Witt vector coordinates $(\xi_0,\xi_1,...)\in R^{\omega}$ is said to be $R_0$-computable if:
\begin{enumerate}
\item For each $n\in \N$, we have that $\xi_n\in R_0$.

\item The function $\N\to R_0:n\mapsto \xi_n$ is recursive.
\end{enumerate}

\end{definition} 



\begin{example}
$(a)$ Let $R_0=R=\F_p$. Then the computable elements of $\Z_p=W(\F_p)$ are the usual computable $p$-adic integers.\\
$(b)$ A \textit{non}-example: Let $R$ be any perfect ring, $R_0\subseteq R$ any computable subring and $S\subseteq \N$ be a non-recursive set. The element $\xi= \sum_{n\in S} p^n\in W(R)$ is not computable.
\end{example}

\subsubsection{Computable untilts}
Let $(F,w)$ be a perfectoid field of characteristic $p$ and $\textbf{A}_{\text{inf}}=W(\Oo_F)$. Let also $R_0\subseteq \Oo_F$ be a computable subring.
\begin{rem}
The reader is welcome to assume that $F=\widehat{\F_p(\!(t)\!)^{1/p^{\infty}}}$ or $\widehat{\overline{\F}_p(\!(t)\!)^{1/p^{\infty}}} $ and $R_0=\F_p[t]$. This case is enough for the applications presented here, i.e. Corollaries \ref{mainCor} and \ref{example}.
\end{rem}

We shall now define what it means for an untilt $K$ of $F$ to be $R_0$-computable:

\begin{definition} \label{defperf}
Let $(K,\iota)$ be an untilt of $(F,w)$ and $R_0\subseteq \Oo_F$ a computable subring. We say that $(K,v)$ is an $R_0$-\textit{computable} untilt of $(F,w)$ if there is an $R_0$-computable distinguished element $\xi_K\in \textbf{A}_{\text{inf}}$ with $\mathcal{O}_K\cong \textbf{A}_{\text{inf}}/(\xi_K)$.
\end{definition}

\begin{example} \label{compuntiltsexample}
In Corollary \ref{compuntilt} we will see that $\widehat{\Q_p(p^{1/p^{\infty}})}$ and $\widehat{\Q_p(\zeta_{p^{\infty}})}$ are $\F_p[t]$-computable untilts of $\widehat{ \F_p(\!(t)\!)^{1/p^{\infty}}}$. Also, that $\widehat{\Q_p^{ab}}$ is an $\F_p[t]$-computable untilt of $\widehat{\overline{ \F}_p(\!(t)\!)^{1/p^{\infty}}}$.
\end{example} 
\subsection{Relative decidability} \label{mainthmsec}
Let $(K,v)$ be an untilt of $(F,w)$ and $\xi_K \in \textbf{A}_{\text{inf}}$ be such that $\Oo_K=\textbf{A}_{\text{inf}}/(\xi_K)$ (Theorem \ref{Font}). We write $\overline{\xi}_{K,n} $ for the image of $\xi_K$ in $W_n(\Oo_F)$ via $\textbf{A}_{\text{inf}}\twoheadrightarrow \textbf{A}_{\text{inf}}/(p^n)\cong W_n(\Oo_F)$.
\subsubsection{Interpretability}
All the necessary background material related to interpretability is presented in \S \ref{intersec}. 
\bl \label{lemainter}
For each $n\in \N$, there is a $\exists^+$-interpretation $\text{A}_n$ of the local ring $(\Oo_K/(p^n),\mathfrak{m}_n)$ in the $\langle L_{\text{rings}}, L_{\text{lcr}} \rangle$-structure $((W_n(\Oo_F), \Oo_F), \overline{\xi}_{K,n})$ such that the sequence of interpretations $(\text{A}_n)_{n\in \omega}$ is uniformly recursive.
\el 
\begin{proof}
Fix $n\in \N$. We have that $\Oo_K/(p^n)\cong \textbf{A}_{\text{inf}}/(p^n,\xi_K) \cong W_n(\mathcal{O}_{F})/(\overline{\xi}_{K,n})$. By Proposition \ref{defnofvaluntilt}, we get that $\mathfrak{m}_n=( \{\theta_n([x]):x\in \mathfrak{m}_F \})$, where $\theta_n$ is the composite map $\textbf{A}_{\text{inf}}\stackrel{\theta} \twoheadrightarrow \textbf{A}_{\text{inf}}/(\xi_K) \stackrel{\mod p^n}\twoheadrightarrow W_n(\Oo_F)/(\overline{\xi}_{K,n}) \xrightarrow{\cong} \Oo_K/(p^n)$. We take $\partial_{\text{A}_n}(x)$ to be $x\in \textbf{W}$. Let $c$ be a constant symbol with $c^{(W_n(\Oo_F),\Oo_F)}=\overline{\xi}_{K,n}$. The reduction map $\text{Form}_{L_{\text{lcr}}}\to \langle L_{\text{rings}}, L_{\text{lcr}} \rangle\cup \{c\}: \phi \mapsto \phi_{\text{A}_n}$ on unnested atomic formulas is described as follows:
\begin{enumerate}
\item If $\phi(x,y)$ is $x=y$, then we take $\phi_{\text{A}_n}$ to be the formula $x,y \in \textbf{W} \land \exists z\in \textbf{W}(x=y+z\cdot c)$. As usual, the conjunct $x,y\in \textbf{W}$ is an informal way of saying that the variables $x,y$ are taken from the sort $\textbf{W}$. The cases of $x=0$ and $x=1$ are similar.

\item If $\phi(x,y,z)$ is $x\diamond y=z$, then we take $\phi_{\text{A}_n}$ to be the formula $x,y,z\in \textbf{W} \land \exists w \in \textbf{W} (x\diamond y=z+w\cdot c )$, where $\diamond$ is either $\cdot$ or $+$.

\item If $\phi(x)$ is $x\in \mathfrak{m}$, then we take $\phi_{\text{A}_n}(x)$ to be the formula $x \in \textbf{W} \land \exists z,w\in \textbf{W}, y\in \textbf{R}  (  y \in \mathfrak{m} \land x=[y]\cdot z+w\cdot c)$. 
\end{enumerate}
The coordinate map $f_{\text{A}_n}:W_n(\Oo_F) \twoheadrightarrow \Oo_K/(p^n)$ is the one induced by the map $\theta_n$ above. The above data define a recursive $\exists^+$-interpretation $\text{A}_n$ of $(\Oo_K/(p^n),\mathfrak{m}_n)$ in $((W_n(\Oo_F),\Oo_F),\overline{\xi}_{K,n})$. The sequence of interpretations $(\text{A}_n)_{n\in \omega}$ is uniform in $n$ and thus also (trivially) uniformly recursive. 
\end{proof}

\bp \label{lemthma}
Suppose $(K,v)$ is an untilt of $(F,w)$ and that $\xi_K\in \textbf{A}_{\text{inf}}$ is such that $\Oo_K=\textbf{A}_{\text{inf}}/(\xi_K)$. Let $(\xi_0,\xi_1,...)\in \Oo_F^{\omega}$ be the Witt vector coordinates of $\xi_K$. Then:\\
$(a)$ The value group $(\Gamma_v,vp)$ is recursively interpretable in the valued field $((F,w),\xi_0)$.\\
$(b)$ For each $n\in \N$, there exists a $\exists^+$-interpretation $\text{B}_n$ of the local ring $(\Oo_K/(p^n),\mathfrak{m}_n)$ in the local ring $((\Oo_F,\mathfrak{m}_F),\xi_0,...,\xi_{n-1})$ such that the sequence of interpretations $(B_n)_{n\in \omega}$ is uniformly recursive.
\ep  
\begin{proof}
$(a)$ By Lemma \ref{lemscholz}$(b)$, we will have that $(\Gamma_v,vp)\cong (\Gamma_w,w \xi_0)$ and the value group $(\Gamma_w,w \xi_0)$ is clearly recursively interpretable in the $L_{\text{val}}\cup \{c\}$-structure $(F,w)$ with $c^{(F,w)}=\xi_0$.\\
$(b)$ Lemma \ref{lemainter} provides us with a $\exists^+$-interpretation $\text{A}_n$ of the local ring $(\Oo_K/(p^n),\mathfrak{m}_n)$ in the $\langle L_{\text{rings}}, L_{\text{lcr}} \rangle \cup \{c\}$-structure $((W_n(\Oo_F), \Oo_F), \overline{\xi}_{K,n})$ such that the sequence of interpretations $(\text{A}_n)_{n\in \omega}$ is uniformly recursive. Lemma \ref{trunclemma} provides us with a quantifier-free interpretation $\Gamma_n$ of $((W_n(\Oo_F),\Oo_F),\overline{\xi}_{K,n})$ in the local ring $((\Oo_F,\mathfrak{m}_F),\xi_0,...,\xi_{n-1})$, such that the sequence of interpretations $(\Gamma_n)_{n\in \omega}$ is uniformly recursive. Let $\text{B}_n$ be the composite interpretation of $(\Oo_K/(p^n),\mathfrak{m}_n)$ in $((\Oo_F,\mathfrak{m}_F),\xi_0,...,\xi_{n-1})$. This is also a $\exists^+$-interpretation by Lemma \ref{transcomplexity} and the sequence of interpretations $(\text{B}_n)_{n\in \omega}$ is uniformly recursive by Lemma \ref{transitivity}.
\end{proof}
\subsubsection{Proof of Theorem \ref{reldec2}}
Given a computable subring $R_0\subseteq \Oo_F$, we write $L_{\text{val}}(R_0)$ for the language $L_{\text{val}}$ enriched with a constant $c_{a}$ for each $a\in R_0$ (see notation). The valued field $(F,w)$ can then be updated to an $L_{\text{val}}(R_0)$-structure with $c_a^{(F,w)}=a$.
\begin{Theor} \label{reldec2}
Let $(F,w)$ be a perfectoid field of characteristic $p$. Suppose $R_0\subseteq \Oo_F$ is a computable subring and $(K,v)$ is an $R_0$-computable untilt of $(F,w)$.\\
$(a)$ If $(F,w)$ is decidable in $L_{\text{val}}(R_0)$, then $(K,v)$ is decidable in $L_{\text{val}}$.\\
$(b)$ If $(F,w)$ is $\exists$-decidable in $L_{\text{val}}(R_0)$, then $(K,v)$ is $\exists$-decidable in $L_{\text{val}}$.
\end{Theor}
\begin{proof}
$(a)$ By Proposition \ref{lemthma}$(a)$, we get that $(\Gamma_v,vp)$ is decidable in $L_{\text{oag}}$ with a constant for $vp$. By Proposition \ref{lemthma}$(b)$, for each $n\in \N$, there exists an interpretation $\text{B}_n$ of the $L_{\text{lcr}}$-structure $(\Oo_K/(p^n),\mathfrak{m}_n)$ in the $L_{\text{lcr}}\cup \{c_0,...,c_{n-1}\}$-structure $((\Oo_F,\mathfrak{m}_F),\xi_0,...,\xi_{n-1})$ with $c_m^{(\Oo_F,\mathfrak{m}_F)}=\xi_m$ for $m=0,...,n-1$. Moreover, the sequence of interpretations $(\text{B}_n)_{n\in \omega}$ is uniformly recursive.  \\
\textbf{Claim:} There exists an interpretation $\Delta_n$ of the $L_{\text{lcr}}\cup \{c_0,...,c_{n-1}\}$-structure $((\Oo_F,\mathfrak{m}_F),\xi_0,...,\xi_{n-1})$ in the $L_{\text{val}}(R_0)$-structure $(F,w)$ such that the sequence of interpretations $(\Delta_n)_{n\in \omega}$ is uniformly recursive.
\begin{proof}
Fix $n\in \N$. We take $\partial_{\Delta_n}(x)$ to be the formula $x\in \Oo$. 
The reduction map on unnested atomic formulas is described as follows:
\begin{enumerate}
\item If $\phi(x)$ is the formula $x=c_m$, for some $m=0,...,n-1$, we take $\phi_{\Delta_n}(x)$ to be the formula $x=c_{\xi_m}$. The formulas $x=y,x=0$ and $x=1$ are interpreted in the obvious way.
\item If $\phi(x)$ is the formula $x\in \mathfrak{m}$, we take $\phi_{\Delta_n}(x)$ to be the formula $\exists y (xy=1\land y\notin \Oo) $. 
\item If $\phi(x,y,z)$ is $x\diamond y=z$, then we take $\phi_{\Delta_n}$ to be the formula $x,y,z\in \Oo \land \phi(x,y,z) $, where $\diamond$ is either $\cdot$ or $+$.
\end{enumerate}
The coordinate map $f_{\Delta_n}: \partial_{\Delta_n}(F)\to \Oo_F$ is the identity. Since $\N\to R_0:m \mapsto \xi_m$ is recursive, the reduction map $\N \times \text{Form}_{L_{\text{lcr}}\cup \{c_0,...,c_{n-1}\}}\to \text{Form}_{L_{\text{val}}(R_0)}:(n,\phi)\mapsto \phi_{\Delta_n}$ restricted to unnested atomic formulas is recursive. By definition, this means that the sequence of interpretations $(\Delta_n)_{n\in \omega}$ is uniformly recursive.
\qedhere $_{\textit{Claim}}$ \end{proof}
For each $n\in \N$, let $E_n$ be the composite interpretation of $\text{B}_n$ and $\Delta_n$. The sequence $(E_n)_{n\in \omega}$ is uniformly recursive by Proposition \ref{transitivity}. By Proposition \ref{uniformdecprop}, we get that the sequence of rings $(\Oo_K/(p^n))_{n\in \omega}$ is \textit{uniformly} decidable in $L_{\text{rings}}$. The conclusion follows from Corollary \ref{vddcor}. \\
$(b)$ For the existential version, one needs to keep track of the complexity of formulas. Since $\text{B}_n$ is a $\exists^+$-interpretation and $\Delta_n$ is an $\exists$-interpretation, we see that the reduction map $L_{\text{lcr}}\to L_{\text{val}}(R_0):\phi\mapsto \phi_{E_n}$ sends $\exists^+$-formulas to $\exists$-formulas (see Remark \ref{existinterrem}). It follows that the sequence of local rings $((\Oo_K/(p^n),\mathfrak{m}_n)_{n\in \omega}$ is uniformly $\exists^+$-decidable in $L_{\text{lcr}}$. We conclude by Corollary \ref{vddcor2}. 
\end{proof}
The assumption of completeness in Theorem \ref{reldec2} can be easily relaxed to henselianity:
\bc \label{reldecgen}
Let $(F,w)$ be a perfect non-trivially valued field with a rank $1$ value group and $R_0\subseteq \Oo_F$ be a computable subring. Suppose $(K,v)$ is a mixed characteristic henselian valued field such that $(\widehat{K},\widehat{v})$ is an $R_0$-computable untilt of $(\widehat{F},\widehat{w})$. Then:\\
$(a)$ If $(F,w)$ is decidable in $L_{\text{val}}(R_0)$, then $(K,v)$ is decidable in $L_{\text{val}}$.\\
$(b)$ If $(F,w)$ is $\exists$-decidable in $L_{\text{val}}(R_0)$, then $(K,v)$ is $\exists$-decidable in $L_{\text{val}}$. 
\ec
\begin{proof}
Suppose $\widehat{\mathcal{O}}_{K}=W(\widehat{\mathcal{O}}_{F})/(\xi)$, where $\xi=[\pi]-up$, with $\pi\in \mathfrak{m}_F \cap R_0$ and $u\in W(\mathcal{O}_{F})^{\times}$. Note that $W(\widehat{\mathcal{O}}_{F})/([\pi]-up) \cong W(\mathcal{O}_F)/([\pi]-up)$; indeed, $W(\mathcal{O}_F)$ is $p$-adically complete and $[\pi]$ and $p$ are associates in the quotient $W(\mathcal{O}_F)/([\pi]-up)$. For each $n\in \N$, one has that
$$\Oo_K/(p^n) \cong \widehat{\mathcal{O}}_{K}/(p^n) \cong W_n(\widehat{\mathcal{O}}_{F})/([\pi]-up) \cong W_n(\mathcal{O}_F)/([\pi]-up)$$
We now proceed as in the proof of Theorem \ref{reldec}.
\end{proof}
\subsubsection{Remarks on Theorem \ref{reldec2}}
Note that decidability of $K$ in $L_{\text{val}}$ relative to its tilt $K^{\flat}$ in $L_{\text{val}}$, i.e. without taking parameters into account, is false. Essentially, the point is that $p$ is named in $L_{\text{val}}$, while $t$ is not: 
\begin{example} \label{lang}
Let $(\Gamma,+,<)\subseteq (\R,+,<)$ be any $p$-divisible ordered abelian group, which is decidable in $L_{\text{oag}}$ but undecidable in $L_{\text{oag}}\cup \{1\}$, where $1$ a distinguished element of $\Gamma$. For example, let $S\subsetneq \mathbb{P}$ be a non-recursive set of primes and $S':=\mathbb{P}-S$ be its complement. Denote by $\frac{1}{S}\Z$ the group generated by $\{\frac{1}{s}:s\in S\}$ and consider the ordered abelian group
$$\Gamma=\frac{1}{p^{\infty}}(\frac{1}{S}\Z \oplus \frac{1}{S'}\Z \sqrt{2})$$
equipped with the order induced from the natural embedding $\Gamma \hookrightarrow \R$. This is a regularly dense group, in the sense of Robinson-Zakon \cite{Rob}, with prime invariants $[\Gamma:p\Gamma]=1$ and $[\Gamma:q\Gamma]=q$ for $q\neq p$. By Robinson-Zakon \cite{Rob}, this group is decidable in $L_{\text{oag}}$ but is clearly undecidable in $L_{\text{oag}}\cup \{1\}$. Since $\Gamma$ is $p$-divisible, we may form the tame valued field $(F,w)=(\mathbb{F}_p(\!(t^{\Gamma})\!),v_t)$, which is decidable in $L_{\text{val}}$ by Theorem 1.4 \cite{Kuhl}. However, the untilt $(K,v)$, whose associated Witt vector is $[t]-p$, is undecidable in $L_{\text{val}}$. Indeed, even $(\Gamma_v,vp)$ is undecidable in $L_{\text{oag}}$ with a constant for $vp$.
\end{example}

\begin{rem} \label{generalthma}
Theorem \ref{reldec} holds also for untilts $K$ of $F$, which have an associated distinguished element $\xi_K$ satisfying a more relaxed condition than the one of Definition \ref{compwittvector}. Namely, suppose that $\xi_K=(\xi_0,\xi_1,...)\in \Oo_F^{\omega}$ and each $\xi_n$ is definable in the valued field $(F,w)$ by a formula $\phi_n(x) \in L_{\text{val}}(R_0)$ with parameters from $R_0$. Suppose furthermore that the function $\N\to \text{Form}_{L_{\text{val}}(R_0)}:n\mapsto \phi_n(x)$ is recursive. Then the conclusion of Theorem \ref{reldec} is still valid, i.e. $(K,v)$ is decidable in $L_{\text{val}}$ relative to $(F,w)$ in $L_{\text{val}}(R_0)$. 
\end{rem}
It is clear that when each $\phi_n(x)$ is a quantifier-free formula, the configuration described in Remark \ref{generalthma} specializes to the notion of a computable untilt. We have no particular application in mind that requires the level of generality described in Remark \ref{generalthma}, which is one reason we have restricted ourselves to the quantifier-free case (the other being clarity of exposition).
\subsection{Corollary \ref{mainsCor}} \label{explicitcomp}
In order to prove Corollary \ref{mainCor}, we need to calculate the tilts of our fields of interest and compute the associated distinguished elements.
\subsubsection{Computing the distinguished elements}
All computations below are well-known to experts (see e.g., Example 2.22 \cite{FF2}). For lack of a detailed reference, we shall spell out the details.
\bl \label{resringab}
There exist ring isomorphisms:\\
$(a)$ $\Z_p[p^{1/p^{\infty}}]/(p) \cong \mathbb{F}_p[t^{1/p^{\infty}}]/(t)$ mapping $p^{1/p^n}+(p)\mapsto t^{1/p^n}+(t)$.\\
$(b)$ $\Z_p[\zeta_{p^{\infty}}]/(p)\cong \mathbb{F}_p[t^{1/p^{\infty}}]/(t^{p-1})$ mapping $\zeta_p+(p)\mapsto t+1+(t^{p-1})$. \\
$(c)$ $\Z_p^{ab}/(p)\cong \overline {\F}_p[t^{1/p^{\infty}}]/(t^{p-1})$ mapping $\zeta_p+(p) \mapsto t+1+(t^{p-1})$.
\el
\begin{proof}
$(a)$ For each $n\in \N$, the irreducible polynomial of $p^{1/p^{n+1}}$ over $\Q_p(p^{1/p^n})$ is the Eisenstein polynomial $x^p-p^{1/p^n}\in \Z_p[p^{1/p^n}][x]$. We therefore compute
$$ \Z_p[p^{1/p^{\infty}}]/(p)\cong \Z_p[x_1,x_2,...]/(p,x_1^p-p,x_2^p-x_1,...)$$ $$\cong \mathbb{F}_p[x_1,x_2,...]/(x_1^p,x_2^p-x_1,...) \stackrel{x_{n}\mapsto t^{1/p^n}}{\cong} \mathbb{F}_p[t^{1/p^{\infty}}]/(t)$$
$(b)$  The irreducible polynomial of $\zeta_p$ over $\Q_p$ is the cyclotomic polynomial $\Phi_p(x)=x^{p-1}+...+1$. Moreover, given $n>1$, the irreducible polynomial of $\zeta_{p^n}$ over $\Q_p(\zeta_{p^{n-1}})$ is $x^p-\zeta_{p^{n-1}}\in \Z[\zeta_{p^n}][x]$. We now proceed as in $(a)$ and compute
$$\Z_p[\zeta_{p^{\infty}}]/(p) \cong \Z_p[x_1,x_2,...]/(p,\Phi_p(x_1),x_2^p-x_1,...) \stackrel{x_{n+1}\mapsto x^{1/p^n}}\cong \F_p[x^{1/p^{\infty}}]/(\overline{\Phi}_p(x)) $$
Note that $\overline{\Phi}_p(x)=\frac{x^p-1}{x-1}=\frac{(x-1)^p}{x-1}=(x-1)^{p-1}$ and thus 
$$\Z_p[\zeta_{p^{\infty}}]/(p)\cong \F_p[x^{1/p^{\infty}}]/(x-1)^{p-1} \stackrel{t=x-1}=\F_p[t^{1/p^{\infty}}]/(t^{p-1})  $$
$(c)$ By local Kronecker-Weber (see Theorem 14.2 in \cite{Wash}) and Proposition 17 in \cite{Ser}, we get that $\Z_p^{ab}=\Z_p^{ur}[\zeta_{p^{\infty}}]$. We now proceed as in $(b)$.
\end{proof}
\bc \label{tiltindeed}
$(a)$ $\widehat{\Q_p(p^{1/p^{\infty}})} ^{\flat}\cong \widehat{\F_p(\!(t)\!)^{1/p^{\infty}}}$ and $t^{\sharp}= p$.\\
$(b)$ $ \widehat{\Q_p(\zeta_{p^{\infty}})}^{\flat}\cong \widehat{\F_p(\!(t)\!)^{1/p^{\infty}}}$ and $(t+1)^{\sharp}= \zeta_p$.\\
$(c)$ $\widehat{\Q_p^{ab}}^{\flat} \cong \widehat{\overline{ \F}_p(\!(t)\!)^{1/p^{\infty}}}$ and $(t+1)^{\sharp}= \zeta_p$.
\ec 
\begin{proof}
$(a)$ 
By Lemma \ref{resringab}, we have that $\Z_p[p^{1/p^{\infty}}]/(p) \cong \mathbb{F}_p[t^{1/p^{\infty}}]/(t)$ via an isomorphism mapping $p^{1/p^n}+(p)\mapsto t^{1/p^n}+(t)$. One can verify directly that  $\widehat{\F_p[\![t]\!]^{1/p^{\infty}}}\to \varprojlim_{\Phi} \mathbb{F}_p[t^{1/p^{\infty}}]/(t) :x\mapsto (x \mod (t), x^{1/p} \mod (t),...) $ is a ring isomorphism. It follows that $\widehat{\Q_p(p^{1/p^{\infty}})} ^{\flat}\cong \widehat{\F_p(\!(t)\!)^{1/p^{\infty}}}$ and by definition $t^{\sharp}=\lim_{n\to \infty} (p^{1/p^n})^{p^n}=p$. The proofs of $(b)$ and $(c)$ are similar.
\end{proof}

\bp [cf. Example 2.22 \cite{FF2}]\label{expl}
We have the following isomorphisms:\\
$(a)$ $\widehat {\Z_p[p^{1/p^{\infty}}]}\cong W( \widehat{\F_p[\![t]\!]^{1/p^{\infty}}})/([t]-p)$.\\
$(b)$ $\widehat{\Z_p[\zeta_{p^{\infty}}]} \cong W( \widehat{\F_p[\![t]\!]^{1/p^{\infty}}})/([t+1]^{p-1}+[t+1]^{p-2}+...+1)$.\\
$(c)$ $\widehat {\Z_p^{ab}}\cong W(\widehat{\overline{  \F}_p[\![t]\!]^{1/p^{\infty}}})/([t+1]^{p-1}+[t+1]^{p-2}+...+1)$.
\ep 
\begin{proof}
$(a)$ By Corollary \ref{tiltindeed}$(a)$, we have that $\widehat{\Q_p(p^{1/p^{\infty}})} ^{\flat}\cong \widehat{\F_p(\!(t)\!)^{1/p^{\infty}}}$. 
Consider the ring homomorphism $\theta: \textbf{A}_{\text{inf}} \twoheadrightarrow \widehat {\Z_p[p^{1/p^{\infty}}]}$ inducing $\sharp: \widehat{\F_p[\![t]\!]^{1/p^{\infty}}}\to \widehat {\Z_p[p^{1/p^{\infty}}]}$. Since $t^{\sharp}=p$, we get that $\theta([t])=p$ and therefore $[t]-p\in \text{Ker}(\theta)$. By Proposition \ref{anydisting}, it follows that $\text{Ker}(\theta)=([t]-p)$. \\
$(b)$ By Corollary \ref{tiltindeed}$(b)$, we have that $ \widehat{\Q_p(\zeta_{p^{\infty}})}^{\flat}\cong \widehat{\F_p(\!(t)\!)^{1/p^{\infty}}}$. 
mapping $1+t\mapsto \zeta_p$. 
Consider the ring homomorphism $\theta: \textbf{A}_{\text{inf}}\twoheadrightarrow \widehat {\Z_p[\zeta_{p^{\infty}}]}$ inducing $\sharp: \widehat{ \F_p[\![t]\!]^{1/p^{\infty}}} \to \widehat {\Z_p[\zeta_{p^{\infty}}]}$. Since $(t+1)^{\sharp}=\zeta_p$, we get that $\theta([t+1])=\zeta_p$ and therefore $[t+1]^{p-1}+....+[t+1]+1\in \text{Ker}(\theta)$. By Proposition \ref{anydisting}, it suffices to show that $\xi=[t+1]^{p-1}+....+[t+1]+1$ is a distinguished element of $\textbf{A}_{\text{inf}}$. Set $\text{res}: \Oo_F \twoheadrightarrow \Oo_F/\mathfrak{m}_F$ for the residue map and $W(\text{res})$ for the unique induced homomorphism $W(\text{res}):W(R)\to W(R/\mathfrak{m}_R)$ of Fact \ref{univwitt}.  We compute that $W(\text{res})(\xi)=1+1+...+1=p$, whence $\xi$ is a distinguished element of $\textbf{A}_{\text{inf}}$ (Remark \ref{twisteddist}).\\ 
$(c)$ Similar to $(b)$. 
\end{proof}




\subsubsection{Proof of Corollary \ref{mainCor}}
\bc \label{compuntilt}
We have the following:\\
$(a)$  $\widehat{\Q_p(p^{1/p^{\infty}})}$ and $\widehat{\Q_p(\zeta_{p^{\infty}})}$ are $\F_p[t]$-computable untilts of $\widehat{\F_p(\!(t)\!)^{1/p^{\infty}}}$.\\
$(b)$ $\widehat{\Q_p^{ab}}$ is an $\F_p[t]$-computable untilt of $\widehat{\F_p(\!(t)\!)^{1/p^{\infty}}}$.
\ec
\begin{proof}
The case of $\widehat{\Q_p(p^{1/p^{\infty}})}$ is clear. For the other two cases, note that $\xi:=[t+1]^{p-1}+[t+1]^{p-2}+...+1$ is $\F_p[t]$-computable, using the \textit{computable} polynomials $S_0,...,S_n$ for Witt vector addition (Observation \ref{compobwitt}).
\end{proof}

\begin{Corol} \label{mainCor}
$(a)$ Assume $\F_p(\!(t)\!)^{1/p^{\infty}}$ is (existentially) decidable in $L_{\text{val}}(t)$. Then $\Q_p(p^{1/p^{\infty}})$ and $\Q_p(\zeta_{p^{\infty}})$ are (existentially) decidable in $L_{\text{val}}$.\\
$(b)$ Assume $\overline {\F}_p(\!(t)\!)^{1/p^{\infty}}$ is (existentially) decidable in $L_{\text{val}}(t)$. Then $\Q_p^{ab}$ is (existentially) decidable in $L_{\text{val}}$.
\end{Corol}
\begin{proof}
$(a)$ By Corollary \ref{compuntilt} and Corollary \ref{reldecgen}.\\
$(b)$ Similar to $(a)$.
\end{proof}
\section{Applications: Tame fields of mixed characteristic}\label{tamesec}

\subsection{Introduction} 

\subsubsection{Motivation} 
We shall prove Corollary \ref{example}, which to my knowledge is the first decidability result for tame fields of mixed characteristic. 
\subsubsection{Preliminaries}
The algebra and model theory of tame fields was introduced and studied by Kuhlmann in \cite{Kuhl}. Recall the definition:
\begin{definition}
Let $(K,v)$ be a henselian valued field. A finite valued field extension $(K',v')/(K,v)$ is said to be tame if 
\begin{enumerate}
\item  $([\Gamma':\Gamma],p)=1$, where $p$ is the \textit{characteristic exponent} of $k$, i.e. $p=\text{char}(k)$ if this is positive and $p=1$ if $\text{char}(k)=0$.
\item The residue field extension $k'/k$ is separable.
\item The extension $(K',v')/(K,v)$ is defectless, meaning that the fundamental equality $[K':K]=[\Gamma':\Gamma]\cdot [k':k]$  holds.
\end{enumerate}
An algebraic valued field extension is said to be tame if every finite subextension is tame.
\end{definition}
\begin{definition}
A henselian valued field $(K,v)$ is said to be tame if every finite valued field extension $(K',v')/(K,v)$ is tame. 
\end{definition}
In practice, one often needs a more intrinsic description of tame fields. This is provided by the following:
\bp [Theorem 3.2 \cite{Kuhl}] \label{tamechar}
Let $(K,v)$ be henselian valued field. Then the following are equivalent:
\begin{enumerate}
\item $(K,v)$ is a tame field.

\item $(K,v)$ is algebraically maximal, $\Gamma$ is $p$-divisible and $k$ is perfect. 
\end{enumerate}
\ep 
\begin{example} \label{tamexample}
Using Proposition \ref{tamechar}, one can verify the following:\\
$(a)$ Any (algebraically) maximal immediate extension of $\Q_p(p^{1/p^{\infty}})$ or $\Q_p(\zeta_{p^{\infty}})$ is tame.\\
$(b)$ Any (algebraically) maximal immediate extension of $\F_p(\!(t)\!)^{1/p^{\infty}}$ is tame. In particular, the Hahn field $ (\F_p(\!(t^{\Gamma})\!),v_t)$ with value group $\Gamma=\frac{1}{p^{\infty}}\Z$ and residue field $\F_p$ is tame.
\end{example}
\subsubsection{Equal characteristic}
Kuhlmann obtained the following Ax-Kochen/Ershov principle for tame fields of equal characteristic: 
\bt [Theorem 1.4 \cite{Kuhl}] \label{thmtamekuhl}
Let $(K,v)$ and $(K',v')$ be two equal characteristic tame fields. Then $(K,v)\equiv (K',v')$ in $L_{\text{val}}$ if and only if $k\equiv k'$ in $L_{\text{rings}}$ and $\Gamma\equiv \Gamma '$ in $L_{\text{oag}}$.
\et  
Kuhlmann then deduces the following decidability result (among others):
\bc [Theorem 1.6 \cite{Kuhl}] \label{hahndec}
Set $\Gamma=\frac{1}{p^{\infty}}\Z$, i.e. for the $p$-divisible hull of $\Z$. The Hahn field $ (\F_p(\!(t^{\Gamma})\!),v_t)$ is decidable in $L_{\text{val}}$.
\ec 
In recent work, Lisinski combined results of Kuhlmann \cite{Kuhl} with work of Kedlaya \cite{Ked} and obtained the following strengthening of Corollary \ref{hahndec}: 
\bt [Theorem 1 \cite{Lis}] \label{listhm}
Set $\Gamma=\frac{1}{p^{\infty}}\Z$. The Hahn field $ (\F_p(\!(t^{\Gamma})\!),v_t)$ is decidable in $L_{t}$.
\et 
\subsubsection{Mixed characteristic} Kuhlmann's Theorem \ref{thmtamekuhl} fails as such for mixed characteristic tame fields. A counterexample is given in Theorem 1.5$(c)$ \cite{KuhlAns}. The lack of such a principle has been a fundamental obstacle in obtaining decidability results for such fields.

\subsection{Mixed characteristic tame fields}
While we do not know whether $\Q_p(p^{1/p^{\infty}})$ and $\Q_p(\zeta_{p^{\infty}})$ are (existentially) decidable, we will show that they admit decidable maximal immediate extensions. These are tame fields by Example \ref{tamexample}$(a)$. 
\bl [Remark 2.23 \cite{FF2}] \label{tiltmaximality}
Let $(K,v)$ be a perfectoid field. Then $(K,v)$ is maximal if and only if $(K^{\flat},v^{\flat})$ is maximal.
\el 
\begin{proof}
Immediate from the tilting equivalence Theorem \ref{tiltingequiv} and the fact that tilting preserves value groups and residue fields (see Lemma \ref{lemscholz}$(b)$).
\end{proof}

\begin{Corol} \label{example}
The valued field $(\Q_p(p^{1/p^{\infty}}),v_p)$ (resp. $(\Q_p(\zeta_{p^{\infty}}),v_p)$) admits a maximal immediate extensions which is decidable in $L_{\text{val}}$.
\end{Corol} 
\begin{proof}
Recall from Corollaries \ref{tiltindeed}$(a)$ and \ref{expl}$(a)$ that $\widehat{\Q_p(p^{1/p^{\infty}})} ^{\flat}\cong \widehat{\F_p(\!(t)\!)^{1/p^{\infty}}}$ and that $\xi=[t]-p$. Let $ \F_p(\!(t^{\Gamma})\!)$ be the Hahn field with value group $\Gamma=\frac{1}{p^{\infty}}\Z$ and residue field $\F_p$. It is a (non-unique) maximal immediate extension of $\F_p(\!(t)\!)^{1/p^{\infty}}$. By Theorem \ref{tiltingequiv}, we may form $K=\F_p(\!(t^{\Gamma})\!)^{\sharp}$ as in the diagram below
\[ \begin{tikzcd}[column sep=2.5em, row sep=2.5em]
K \arrow[dash]{d}{}  & \F_p(\!(t^{\Gamma})\!) \arrow[dash]{d}{} \arrow[l, dashed, bend right=10]{}[swap]{\sharp}\\%
\widehat{\Q_p(p^{1/p^{\infty}})} \arrow[r, "\flat"] & \widehat{\F_p(\!(t)\!)^{1/p^{\infty}}}
\end{tikzcd}
\]
such that $\xi_K=[t]-p$ (see \S \ref{tiltingequivsec}). By Lemma \ref{tiltmaximality}, we see that $(K,v)$ is a maximal immediate extension of $(\Q_p(p^{1/p^{\infty}}),v_p)$. 
By Theorem \ref{listhm} and Theorem \ref{reldec2}, we see that $(K,v)$ is decidable in $L_{\text{val}}$. The proof for $\Q_p(\zeta_{p^{\infty}})$ is similar.
\end{proof}
\begin{rem}
It is also true that $\Q_p^{ab}$ admits a decidable maximal immediate extension (in this case unique) but this already follows  from well-known results in the model theory of algebraically maximal Kaplansky fields (see pg.4 Part (f) \cite{Kuhl}). 
\end{rem}
In spite of Corollary \ref{example}, the following is worth noting: 
\begin{rem} \label{mostmaxundec}
A tree-like construction similar to Proposition \ref{scholzprop}$(b)$ shows that there exist uncountably many pairwise elementary inequivalent maximal immediate extensions of $\Q_p(p^{1/p^{\infty}})$ (resp. $\Q_p(\zeta_{p^{\infty}})$). In particular, the valued field $\Q_p(p^{1/p^{\infty}})$ (resp. $\Q_p(\zeta_{p^{\infty}})$) has uncountably many \textit{undecidable} maximal immediate extensions. Note that the tilts of all such fields will be maximal immediate extensions of $\F_p(\!(t)\!)^{1/p^{\infty}}$ and thus will be tame fields. As a consequence of Kuhlmann's Theorem \ref{thmtamekuhl}, all of them will be decidable in $L_{\text{val}}$. Nevertheless, they will all be undecidable in $L_{\text{val}}(t)$.
\end{rem}


\section{Applications: Congruences modulo $p$} \label{cong}
\subsection{Introduction}
\subsubsection{Goal}We shall now prove Theorem \ref{congmodp}, which shows the existence of an algorithm that decides whether a system of polynomial equations and inequations, defined over $\Z$, has a solution modulo $p$ over the valuation rings of our fields of interest. 
\subsubsection{Strategy}
The proof is via a local field approximation argument, using the computations of Lemma \ref{resringab}. One eventually encodes the above problem in the existential theory of the tilt in $L_{\text{val}}$, where the Anscombe-Fehm Theorem \ref{AnsFehm} applies. This should be contrasted with Corollary \ref{mainCor}, which requires decidability in the language $L_{\text{val}}(t)$ on the characteristic $p$ side. 
\subsection{From residue rings to valuation rings }
The crux of the argument lies in the following:
\bp \label{localglobal}
Let $f_i(x),g_j(x)\in \F_p[x]$ be multi-variable polynomials in $x=(x_1,...,x_m)$ for $i,j=1,...,n$. Then
$$  \F_p[t^{1/p^{\infty}}]/(t) \models \exists x \bigwedge_{1\leq i,j\leq n} (f_i(x)=0\land g_j(x)\neq 0) \iff $$ 
$$ \F_p[\![t]\!]^{1/p^{\infty}}\models \exists x\bigwedge_{1\leq i,j\leq n}(v(f_i(x))> v(g_j(x)))$$
\ep  

\begin{proof} 
First observe that
$$ \F_p[\![t]\!]^{1/p^{\infty}}/(t)\cong \varinjlim  \F_p[\![t^{1/p^n}]\!]/(t)\cong \varinjlim  \F_p[t^{1/p^n}]/(t)\cong  \F_p[t^{1/p^{\infty}}]/(t) \ \ (\dagger)$$
"$\Rightarrow$": Let $a\in ( \F_p[t^{1/p^{\infty}}]/(t))^m$ be such that $f_i(a)=0\land g_j(a)\neq 0 $, for $1\leq i,j\leq n$ and let $\tilde{a}$ be any lift of $a$ in $ \F_p[\![t]\!]^{1/p^{\infty}}$ via the isomorphism $(\dagger)$. We see that $v(f_i(\tilde{a}))\geq 1 >v(g_j(\tilde{a}))$, for all $1\leq i,j\leq n$.\\
"$\Leftarrow$": Let $b\in ( \F_p[\![t]\!]^{1/p^{\infty}})^m$ be such that $v(g_j(b))< v(f_i(b))$ for all $1\leq i,j\leq n$. Set $\gamma_1=\max \{v(g_j(b)):j=1,...,n\}$ and $\gamma_2=\min \{v(f_i(b)):i=1,...,n\}$ and consider the open interval $I=(\gamma_1,\gamma_2)\subseteq \frac{1}{p^{\infty}}\Z^{\geq 0}$. Since $\frac{1}{p^{\infty}}\Z$ is dense in $\mathbb{R}$, we can find $q\in \frac{1}{p^{\infty}}\Z$ such that $1\in qI$. 

We now make use of the fact that for each $q\in \frac{1}{p^{\infty}}\Z^{>0}$, there is an embedding 
$$\rho:  \F_p(\!(t)\!)^{1/p^{\infty}}\to \F_p(\!(t)\!)^{1/p^{\infty}}$$
which maps $t\mapsto t^q$. Indeed, if $q\in \frac{1}{p^{N}}\Z^{>0}$ for some $N\in \N$, then there exists an embedding $\rho: \F_p(\!(t)\!)^{1/p^{N}}\to  \F_p(\!(t)\!)^{1/p^{N}}$ mapping $t\mapsto t^q$, exactly as in Remark 7.9 \cite{AnscombeFehm}. Such a map can also be extended uniquely to the perfect hull $  \F_p(\!(t)\!)^{1/p^{\infty}}$.\\
Now let $\rho:   \F_p(\!(t)\!)^{1/p^{\infty}}\to  \F_p(\!(t)\!)^{1/p^{\infty}}$ be as above. Then, since $f_i(x),g_j(x)\in  \F_p[x]$, we get
$$v(g_j(\rho(b)))=v(\rho(g_j(b)))=qv(g_j(b))<qv(f_i(b))=v(\rho(f_i(b)))=v(f_i(\rho(b)))$$
for all $1\leq i,j\leq n$. We may thus replace our witness $b$ with $a=\rho(b)$.

Since $1\in qI$, we get $f_i(a)=0\mod (t) \land g_j(a)\neq 0\mod (t) $, for all $i,j=1,...,n$. The reduction of $a$ modulo $(t)$, seen as a tuple in $  \F_p[t^{1/p^{\infty}}]/t$ via $(\dagger)$, is the desired witness.
\end{proof}
\begin{rem} \label{modif}
The same argument used in Lemma \ref{localglobal} shows that 
$$ k[t^{1/p^{\infty}}]/(t^{p-1})\models \exists x \bigwedge_{1\leq i,j\leq n} (f_i(x)=0\land g_j(x)\neq 0) \iff $$ 
$$ k[\![t]\!]^{1/p^{\infty}}\models \exists x\bigwedge_{1\leq i,j\leq n}(v(f_i(x))> v(g_j(x))) \iff $$
where $k=\F_p$ or $\overline {\F}_p$. The argument of Proposition \ref{localglobal} needs only a slight modification for the converse direction; one needs to take $q\in \frac{1}{p^{\infty}}\Z$ such that $(p-1)\in qI$ instead of $1\in qI$ and the same proof goes through.
\end{rem}
\subsection{Proof of Theorem \ref{congmodp}}
We may now prove the following:
\begin{Theor} \label{congmodp}
Let $K$ be any of the fields $\Q_p(p^{1/p^{\infty}}),\Q_p(\zeta_{p^{\infty}})$ or $\Q_p^{ab}$. Then the existential theory of $\mathcal{O}_K/(p)$ is decidable in $L_{\text{rings}}$.
\end{Theor}
\begin{proof}
By the Anscombe-Fehm Theorem \ref{AnsFehm}, the valued fields $(\F_p(\!(t)\!)^{1/p^{\infty}},v_t)$ and $(\overline {\F}_p(\!(t)\!)^{1/p^{\infty}},v_t)$ are $\exists$-decidable in $L_{\text{val}}$. For $\Q_p(p^{1/p^{\infty}})$ the conclusion now follows from Lemma \ref{resringab}$(a)$ and Proposition \ref{localglobal}. For the other two fields, one has to use Lemma \ref{resringab}$(b),(c)$ and Remark \ref{modif}.
\end{proof}
\begin{rem}
Note that Corollary \ref{mainsCor} requires decidability in the language $L_{\text{val}}(t)$ on the characteristic $p$ side. However, for the purposes of Theorem \ref{congmodp}, the Anscombe-Fehm results in $L_{\text{val}}$ turned out to be sufficient. This became possible because of Proposition \ref{localglobal}, which "eliminates" any reference to $t$.
\end{rem}

\section{Final remarks} \label{final}

\subsection{An almost decidable field}
In \S \ref{uniformitydecidab}, we emphasized that \textit{uniform} decidability of $(\Oo_K/(p^n))_{n\in \omega}$ is key for the decidability of $(K,v)$. Indeed, at least by assuming the decidability of $\widehat{\F_p(\!(t)\!)^{1/p^{\infty}}}$ in $L_{\text{val}}(t)$, Proposition \ref{scholzprop}$(b)$ allows us to produce undecidable valued fields $(K,v)$ with each individual residue ring $\Oo_K/(p^n)$ being decidable.
\bp \label{almostdec}
Assume $\widehat{\F_p(\!(t)\!)^{1/p^{\infty}}}$ is decidable in $L_{\text{val}}(t)$. Then there exists a mixed characteristic henselian valued field $(K,v)$ such that:
\begin{enumerate}
\item The valued field $(K,v)$ is undecidable in $L_{\text{val}}$.
\item For each $n\in \N$, the ring $\Oo_K/(p^n)$ is decidable in $L_{\text{rings}}$.
\item The value group $(\Gamma_v,vp)$ decidable in $L_{\text{oag}}$. 

\end{enumerate}

\ep 

\begin{proof}
For the convenience of the reader, we first review the construction from proof of Proposition \ref{scholzprop}$(b)$. Given $\alpha \in 2^{\omega}$, define inductively:
\begin{enumerate}
\item $K_0=\Q_p$ and $\pi_0=p$.
\item $K_{\alpha\restriction n}=K_{\alpha\restriction (n-1)}(((1+p)^{\alpha(n-1)}\cdot \pi_{\alpha\restriction (n-1)})^{1/p} )$ and $\pi_{\alpha\restriction n}=((1+p)^{\alpha(n-1)}\cdot \pi_{\alpha\restriction (n-1)})^{1/p}$. 
\end{enumerate}
We let $K_{\alpha}=\bigcup_{n\in \omega} K_{\alpha \restriction n}$. Set $\overline{\alpha}_n=\sum_{k=0}^{n-1} \alpha(k)\cdot p^k$. Claim 3 of Proposition \ref{scholzprop} shows that $\widehat{K_{\alpha}}^{\flat}= \widehat{\F_p(\!(t)\!)^{1/p^{\infty}}}$. We shall argue below that any $\widehat{K_{\alpha}}$ satisfies conditions (2) and (3). On the other hand, since $\widehat{ K_{\alpha}}\not \equiv \widehat{K_{\beta}}$ for $\alpha \neq \beta$ (see Claim 2, Proposition \ref{scholzprop}$(b)$), we will have that some $\widehat{K_{\alpha}}$ is undecidable in $L_{\text{val}}$. 

Fix $\alpha \in 2^{\omega}$, set $\overline{a}= \sum_{k\geq 0} \alpha(k)\cdot p^k\in \Z_p$ and $\overline{\alpha}_n= \sum_{k=0}^{n-1} \alpha(k)\cdot p^k \in \Z$ for all $n\in \N^{>0}$ ($\overline{\alpha}_0=0$ by convention). One sees that $(1+p)^{p^n}\equiv 1 \mod p^{n+1} \Z_p$, for all $n\in \N$. It follows that the limit $\lim_{n\to \infty} (1+p)^{\overline{\alpha}_n}$ exists in $\Z_p$ and we denote it by $(1+p)^{\overline{a}}$. Fix $n\in \N^{> 0}$ for the rest of the proof.\\
\textbf{Claim: } We have $t^{\sharp}= (1+p)^{\overline{a}} \cdot p$ and  $t^{\sharp} \equiv (1+p)^{\overline{\alpha}_{n-1}}\cdot p \mod p^{n} \Oo_{K_{\alpha}}$.
\begin{proof}
We have that $t=(\pi_{\alpha \restriction 0}+(p),\pi_{\alpha \restriction 1}+(p),...)$, via the identification $\widehat{\F_p[\![t]\!]^{1/p^{\infty}}}\cong \varprojlim_{\Phi} \Oo_{K_{\alpha}}/(p)$ (see Remark \ref{claim2t+1}). Since $\pi_{\alpha\restriction n}^{p^n}=(1+p)^{\overline{\alpha}_n}\cdot  p$, we compute $t^{\sharp}=\lim_{n\to \infty} \pi_{\alpha(n)}^{p^n}= \lim_{n\to \infty}(1+p)^{\overline{\alpha}_n} \cdot p=(1+p)^{\overline{a}} \cdot p$. Since $(1+p)^{p^{n-1}}\equiv 1 \mod p^{n} \Oo_{K_{\alpha}}$, we also get $t^{\sharp} \equiv (1+p)^{\overline{\alpha}_{n-1}}\cdot p \mod p^{n} \Oo_{K_{\alpha}}$. 
\qedhere $_{\textit{Claim}}$ \end{proof}
As a consequence, we may take $\xi_{\alpha} = [t]-(1+p)^{\overline{a}} \cdot p$ as a generator of the ideal of $\textbf{A}_{\text{inf}}$ associated to $K_{\alpha}$. Consider also the distinguished element $\xi':=[t]-(1+p)^{\overline{\alpha}_{n-1}}\cdot p \in \textbf{A}_{\text{inf}}$ and the associated untilt $(K',v')$. We have an equality of ideals $(p^n,\xi_{\alpha})=(p^n,\xi')$ in $\textbf{A}_{\text{inf}}$ and thus an isomorphism of rings $\Oo_{K_{\alpha}}/(p^n)\cong \Oo_{K'}/(p^n)$.

Having assumed that the valued field $(\widehat{\F_p(\!(t)\!)^{1/p^{\infty}}},v_t)$ is decidable in $L_{\text{val}}(t)$ and since $\xi'$ is $\F_p[t]$-computable, we get that the valued field $(K',v')$ is decidable in $L_{\text{val}}$ by Theorem \ref{reldec}. In particular, the ring $\Oo_{K_{\alpha}}/(p^n)\cong \Oo_{K'}/(p^n)$ is decidable in $L_{\text{rings}}$. The value group $(\Gamma_{\alpha},v_{\alpha} p)\cong (\frac{1}{p^{\infty}}\Z,1)$ is also decidable in $L_{\text{oag}}$ with a constant symbol for $1$. Since $n$ was arbitrary, this concludes the proof.
\end{proof}
\begin{rem}
It seems plausible that a similar construction can be carried out with $\F_p(\!(t^{\Gamma})\!)$ instead of $\widehat{\F_p(\!(t)\!)^{1/p^{\infty}}}$, where $\Gamma=\frac{1}{p^{\infty}}\Z$, thereby recovering the above result unconditionally. Nevertheless, we do not have a working example at present.
\end{rem}

\subsection{Reversing the direction} \label{rev}
\subsubsection{} Given that our understanding of decidability problems in characteristic $p$ is limited, our philosophy of reducing decidability questions from mixed characteristic to positive characteristic may seem impractical. Nevertheless, we have already seen two applications in \S \ref{tamesec} and \S \ref{cong}. We also mention another application in \cite{KK3}, which proves an \textit{undecidability} result for the asymptotic theory of $\{ K: [K:\Q_p]<\infty\}$ in the language $L_{\text{val}}$ with a cross-section, again by transposing a result in positive characteristic.
\subsubsection{ }
We shall now demonstrate that the characteristic $p$ difficulties in the language $L_{\text{val}}(t)$ are already encoded in the mixed characteristic setting, by showing a relative decidability result in the opposite direction:
\bp \label{from0top}
If $\Q_p(p^{1/p^{\infty}})$ is $\forall^1 \exists$-decidable in $L_{\text{val}}$, then $ \F_p[\![t]\!]^{1/p^{\infty}}$ is $\exists^+$-decidable in $L_{\text{val}}(t)$. 
\ep 
\begin{proof}
We may focus on $L_{\text{rings}}(t)$ sentences since $x\in \Oo$ is $\exists^+$-definable in $L_{\text{rings}}(t)$ (cf. footnote \ref{definabilityfoot}). Let $f_i(X_1,...,X_m,T) \in \F_p[X_1,...,X_m,T]$ for $i=1,...,n$. We claim that 
$$  \F_p[\![t]\!]^{1/p^{\infty}} \models \exists x_1,...,x_m(\bigwedge_{1\leq i\leq n} f_i(x_1,...,x_m,t)=0)\iff $$
$$ \F_p[t^{1/p^{\infty}}]/(t)\models \forall y\in \mathfrak{m} \exists x_1,...,x_m(\bigwedge_{1\leq i\leq n} f_i(x_1,...,x_m,y)=0)$$
It is enough to prove the claim since $\Z_p[p^{1/p^{\infty}}]/(p)\cong \F_p[t^{1/p^{\infty}}]/(t)$ by Lemma \ref{resringab}$(a)$. We leave it to the reader to write down the $\forall^1\exists$-statement about $\Q_p(p^{1/p^{\infty}})$ which is equivalent to the one about $\F_p[t^{1/p^{\infty}}]/(t)$ written above.\\
$"\Rightarrow":$ Let $(\alpha_1,...,\alpha_m)\in (\F_p[\![t]\!]^{1/p^{\infty}})^m$ be such that 
$$ f_1(\alpha_1,....,\alpha_m,t)=f_2(\alpha_1,....,\alpha_m,t)=...=f_n(\alpha_1,....,\alpha_m,t)=0$$ 
Let $y\in \mathfrak{m}\subset   \F_p[\![t]\!]^{1/p^{\infty}}$ and $ \kappa \in \N$ be such that $y^{p^{\kappa}}\in  \F_p[\![t]\!]$. Then there exists a ring endomorphism $\rho$ on $ \F_p[\![t]\!]$ mapping $t\mapsto y^{p^{\kappa}}$, which extends uniquely to $  \F_p[\![t]\!]^{1/p^{\infty}}$. Let $\beta_i:=\rho(\alpha_i)$ for $i=1,...,m$. Since $f_i(X_1,...,X_m,T) \in \F_p[X,T]$, we get 
$$f_1(\beta_1,...,\beta_m,y^{p^{\kappa}})=f_2(\beta_1,...,\beta_m,y^{p^{\kappa}})=...=f_n(\beta_1,...,\beta_m,y^{p^{\kappa}})=0 $$ 
and thus 
$$f_1(\beta_1^{1/p^{\kappa}},...,\beta_m^{1/p^{\kappa}},y)=f_2(\beta_1^{1/p^{\kappa}},...,\beta_m^{1/p^{\kappa}},y)=...=f_n(\beta_1^{1/p^{\kappa}},...,\beta_m^{1/p^{\kappa}},y)=0$$ 
In particular, the tuple $(x_1,...,x_m):=(\beta_1^{1/p^{\kappa}},...,\beta_m^{1/p^{\kappa}})$ is a solution modulo $(t)$ to the above system of equations.\\
$"\Leftarrow"$ A generalized version of Greenberg's Theorem due to Moret-Bailly (see Corollary 1.2.2 \cite{LMB}) shows that
$$f_1(x_1,...,x_m,t)=f_2(x_1,...,x_m,t)=...=f_n(x_1,...,x_m,t)=0$$ 
has a solution in $  \F_p[\![t]\!]^{1/p^{\infty}}$ if and only if it has a solution modulo $(t^N)$, for all $N\in \N$. Equivalently, if and only it has a solution modulo $(t^{p^N})$ for all $N\in \N$. Using the $N$-th iterated Frobenius, we see that for any $N\in \N$ we have
$$   \F_p[\![t]\!]^{1/p^{\infty}}\models \exists x_1,...,x_m(\bigwedge_{1\leq i\leq n} f_i(x_1,...,x_m,t)=0 \mod t^{p^N})\iff$$
$$   \F_p[\![t]\!]^{1/p^{\infty}}\models \exists x_1,...,x_m (\bigwedge_{1\leq i\leq n} f_i(x_1,...,x_m,t^{1/p^N})=0 \mod t)$$
and the latter is true by assumption, for any $N\in \N$.
\end{proof}
\begin{rem}
By an identical argument, one can prove a similar result for $\Q_p(\zeta_{p^{\infty}})$ or $\Q_p^{ab}$. In the case of the latter, one gets that $\overline {\F}_p[\![t]\!]^{1/p^{\infty}}$ is $\exists^+$-decidable in $L_{\text{val}}(t)$, provided that $\Q_p^{ab}$ is $\forall^1 \exists$-decidable in $L_{\text{val}}$.
\end{rem}
It would be nice to have a version of Proposition \ref{from0top} for the full theories. Together with Corollary \ref{mainsCor}, this would yield a Turing equivalence between the theories of $\Q_p(p^{1/p^{\infty}})$ and $\F_p(\!(t)\!)^{1/p^{\infty}}$. Nevertheless, Proposition \ref{from0top} still suggests that if we eventually want to understand the theories of $\Q_p(p^{1/p^{\infty}})$, $\Q_p(\zeta_{p^{\infty}})$ and $\Q_p^{ab}$ (even modest parts of their theories), we would have to face certain characteristic $p$ difficulties, posed in Question \ref{question} below:
\bq \label{question}
$(a)$ Is $\text{Th}_{\exists^+}(\F_p[\![t]\!]^{1/p^{\infty}})$ decidable in $L_{\text{val}}(t)$?\\
$(b)$ Is $\text{Th}_{\exists^+}(\overline {\F}_p[\![t]\!]^{1/p^{\infty}})$ decidable in $L_{\text{val}}(t)$?
\eq

\subsection*{Acknowledgements}
A great mathematical debt is owed to  E. Hrushovski from whom I learned many of the crucial ideas in this work. I would like to thank J. Koenigsmann for insightful suggestions and his encouragement throughout. Many thanks to T. Scanlon for his hospitality while I was a visitor at UC Berkeley and for many inspiring discussions. I also thank all participants in a study group that took place at Fields Institute in Fall 2021, for their meticulous readings of earlier versions of this work and for pointing out several local corrections. I finally wish to thank two anonymous referees for their thorough reviews.

\bibliographystyle{alphaurl}
\bibliography{references2}
\Addresses
\end{document}